\documentclass[twocolumn, amsthm]{autart}    

\usepackage{amsmath}

\usepackage{graphicx}
\usepackage{cite}

\usepackage{amsmath,amssymb, amsfonts,bbm}
\usepackage{xcolor}
\usepackage{algorithm}
\usepackage{algpseudocode}

\algsetblockdefx[NoIndentFor]{NoIndentFor}{EndNoIndentFor}%
  {}{0pt}%
  [1]{\algorithmicfor\ #1\ \algorithmicdo}%
  {\algorithmicend\ \algorithmicfor}
\usepackage{dblfloatfix}

\newcounter{MYtempeqncnt}

\newtheorem{theorem}{Theorem}
\newtheorem{problem}{Problem}
\newtheorem{remark}{Remark}
\newtheorem{lemma}{Lemma}
\newtheorem{assumption}{Assumption}
\newtheorem{definition}{Definition}

\begin{document}

\begin{frontmatter}

\title{A No-Regret Framework for Adaptive Incentive Design} 

\author[KTH]{Georgios Vasileiou}\ead{geovas@kth.se}, 
\author[KTH]{Lantian Zhang}\ead{lantian@kth.se},       
\author[KTH]{Silun Zhang}\ead{silunz@kth.se}

\address[KTH]{Department of Mathematics, KTH Royal Institute of Technology, Stockholm, 10044, Sweden.}


\begin{abstract}                         
Incentive design studies how a central authority can influence strategic agents through payments, subsidies, or taxes, so that individual objectives align with collective welfare. This paper introduces a No-Regret Adaptive Incentive Design (RAID) framework for nonlinear games with continuous action spaces and private agent costs. In this framework, the authority (planner) designs incentives that regulate the Nash equilibrium toward a socially optimal action profile, while simultaneously learning agents’ unknown preferences from repeated strategic responses. We formulate the RAID problem and construct a least-squares estimator whose strong consistency requires only diminishing excitation. 
Leveraging this weak excitation requirement, we propose a switching incentive policy that alternates between probing (exploration) and estimate-based (exploitation) incentives. The resulting policy achieves an $O(t^{-0.5})$ parameter estimation rate and accumulates $O(t^{0.5}\log t)$
squared social-cost regret, almost surely.
We further extend the framework to an endogenous-noise response model, where standard least-squares estimation is biased due to an error-in-variables correlation between the noise and agent responses.
We utilize a repeated-sampling estimator and corresponding switching policy that retain the same almost-sure convergence and regret rates. 
Numerical experiments validate the effectiveness and 
predicted convergence rates of the method.
\end{abstract}

\end{frontmatter}

\section{Introduction}\label{sec:Introduction}
Modern infrastructure often interacts with strategic participants. Electricity consumers decide when to use power, drivers choose routes, and data owners determine what information to contribute.  In each case, a central authority cannot directly prescribe the participants' actions but instead influences them through incentive signals such as pricing, subsidies, tolls, and taxes. \textit{Incentive design} studies how such signals should be chosen to align individual strategic behavior with system-level objectives. This problem is central to a wide range of applications, including adaptive pricing in demand-response markets~\cite{satchidanandan23atwostage, liAdaptivePricingOptimal2025}, promoting eco-driving~\cite{umar26ecodriving,povedaDistributedAdaptivePricing, barreraDynamicIncentives, grootSystemOptimalRoutingTraffic2015},
and contract pricing in data markets~\cite{bonatti24sellinginfo}. 

Incentive design is inherently hierarchical. A principal first commits to an incentive rule, agents then choose actions that optimize their local objectives, and the principal subsequently implements the corresponding subsidies or taxes. This structure resembles a Stackelberg game, but with a crucial distinction: the principal does not commit to a single action. Instead, the principal commits to a \textit{policy}---a map from agents' actions to monetary transfers. This distinction defines the class of reverse or inverse Stackelberg games~\cite{grootReverseStackelbergGames2012, grootReverseStackelbergGames2012a, olsderPhenomenaInverseStackelberg2009, olsderPhenomenaInverseStackelberg2009a}. From a control-theoretic perspective, incentive design is an equilibrium-regulation problem. A feedback incentive policy maps observed actions into payments, thereby shaping the game faced by the agents and steering its equilibrium toward a desired operating point~\cite{hoControltheoreticViewIncentives1980, basarAffineIncentiveSchemes1984}. 

Beyond the reverse Stackelberg formulation, incentive design is especially concerned with scenarios of \textit{information asymmetry}. This information asymmetry may arise because actions are imperfectly observed, a setting known as moral hazard, or because the principal does not know the agents' preferences, a setting known as adverse selection (see \cite{picardDesignIncentiveSchemes1987, ratliffPerspectiveIncentiveDesign2019} and the references therein). We focus particularly on the latter case. If agents' costs were known, the principal could use monetary transfers to reshape their effective objectives and implement an equilibrium consistent with the social optimum. In practice, however, agents' preferences are often private. Preference learning and equilibrium regulation then become coupled. In this case, the principal needs to infer agents' preferences while simultaneously choosing incentives to regulate a more desirable action~\cite{ratliffPerspectiveIncentiveDesign2019,ratliffAdaptiveIncentiveDesign2021}. When the principal's objective depends only on agents' actions, a Social-Gradient method is proposed to bypass preference learning~\cite{vasileiou2026incentivedesignhypergradientssocialgradient}. Work \cite{maheshwariAdaptiveIncentiveDesign2024} provides a two-timescale incentive scheme when the externality is observable.

This work develops a no-regret framework for incentive design in continuous action spaces, in which the principal learns agents' preferences while regulating the equilibrium induced by their strategic responses. The framework relaxes standard excitation requirements and quantifies the exploration-exploitation tradeoff in repeated principal-agent interactions. The resulting closed-loop incentives are robust to model uncertainty and measurement noise, and steer the incentive-action pair toward the social optimum. Moreover, the proposed framework achieves sublinear regret almost surely relative to a perfectly informed oracle.

\subsection{Related Works}\label{sec:Intro_related_work}

\subsubsection{Stackelberg, Reverse Stackelberg, and Control-theoretic Incentives}
Classical Stackelberg games provide the standard model of leader-follower decision-making~\cite{fudenbergGameTheory1991}. Reverse (or inverse) Stackelberg games extend this model by allowing the leader to announce a response-dependent policy rather than a single action~\cite{grootReverseStackelbergGames2012, grootReverseStackelbergGames2012a, olsderPhenomenaInverseStackelberg2009, olsderPhenomenaInverseStackelberg2009a}. 
By optimizing over policies, the principal gains additional leverage to steer equilibrium behavior~\cite{zhangSolutionStackelbergMultifollower, calderonePricingCoordinationOpen2014} and to infer hidden information~\cite{hoInformationStructure1981}. This policy viewpoint is closely related to the control-theoretic view of incentives, introduced by Ho et al.~\cite{hoControltheoreticViewIncentives1980}, in which incentives act as feedback mechanisms for coordinating strategic decision makers. This framework was subsequently developed through study of incentive controllability~\cite{hoControltheoreticViewIncentives1980, zhengExistenceDerivationOptimal1982a, changConceptInducibleRegion1982}, feedback and closed-loop incentive structures~\cite{basarClosedloopStackelbergStrategies1979a, tolwinskiClosedloopStackelbergSolution1981}, multilevel systems~\cite{cruzLeaderfollowerStrategiesMultilevel1978, simaanStackelbergSolutionGames1973}, and incentive credibility~\cite{hoCredibilityRationalityPlayers1982, luhCredibilityStackelbergGames1984}. Ho et al.~\cite{hoInformationStructure1981} have also investigated the asymmetric information structure that occurs in stochastic incentive problems, a direction that was expanded upon in~\cite{luhCredibilityStackelbergGames1984, xiaopingliuOptimalIncentiveStrategy1992, tuPerformanceInformativenessLinearquadratic1988,chen2025activeinversemethodsstackelberg}. This work builds on this foundational control-theoretic literature. Rather than characterizing solutions to a known Stackelberg or reverse Stackelberg game, we investigate nonlinear games whose utilities are unknown to the principal.  The resulting problem has an adaptive estimation-and-control structure, in which the policy must elicit informative agent responses for preference learning while regulating those responses toward the social optimum.

\subsubsection{Adaptive Incentives and Stackelberg Learning}
Recent work has increasingly utilized learning-based formulations of repeated Stackelberg games to develop adaptive incentive mechanisms. 
Works such as~\cite{balcanCommitmentRegretsOnline2015a, laufferNoRegretLearningDynamic2024} describe adaptive policies accompanied by no-regret guarantees for mixed (probabilistic) strategies over games with finite action and type spaces.  Equilibrium steering in Bayesian normal-form games with bi-matrix payoffs is investigated in~\cite{yorulmazSoftInducementFramework2025}, where the principal guides learning agents toward desired action profiles with sublinear regret. 
In this context, information asymmetry is regarded as a stochastic variable that informs the utility of the acting player~\cite{sessaContextualGamesMultiAgent2021}. Building on this viewpoint,
~\cite{harrisRegretMinimizationStackelberg2024}
shows that no-regret learning is impossible in a repeated Stackelberg setting with adversarial context, yet sublinear regret may be achieved under a stochastic context or stochastic agent assumption. By contrast, the setting studied here is a continuous-action nonlinear game with an unknown continuous type matrix appearing in the agents' marginal costs. The principal not only optimizes its own objective but also identifies the agents' cost gradients from equilibrium responses and uses the estimates to synthesize incentives. The resulting problem is closer to adaptive control and stochastic regression~\cite{zhang2026stochasticadaptivecontrolsystems,vanweerelt2025selfidentifyinginternalmodelbasedonline,zhang2025onlinelearningnonlineardynamical}, but with the regressors generated endogenously by an incentive-induced Nash equilibrium.  

\subsubsection{Utility Learning and Inverse Games}
Adaptive incentive design is also closely related to the objective of \textit{utility learning}, which aims to infer latent preferences from observed behavior. In economics, this objective is often encountered in revealed preference analyses~\cite{varianRevealedPreference2006, varianNonparametricApproachDemand1982, afriatConstructionUtilityFunctions1967}, and extended to the multi-agent setting through inverse game theory~\cite{kuleshovInverseGameTheory2015, liseEstimatingGameTheoretic2001} and inverse reinforcement learning~\cite{ngAlgorithmsInverseReinforcementLearning}.  The key distinction from utility learning is that the principal actively influences the informativeness and bias of the observed data. Consequently, the principal's ability to recover utility functions hinges on actively designing informative experiments rather than learning from fixed data that describe equilibrium conditions.

\subsection{Contributions}

We introduce No-Regret Adaptive Incentive Design (RAID), in which the principal learns agents’ preferences while regulating strategic behavior toward the social optimum. Our main contributions are as follows. 

First, we investigate incentive design for an $n$-agent nonlinear game and show that, under the diagonal strict convexity condition, an affine incentive is both sufficient 
and optimal within the class of convex incentives that depend only on each agent's own action (Lemma \ref{lem:linear_inccentive_is_optimal}). We further establish that, under such an incentive, the induced response map remains unique and is a diffeomorphism from an (unknown) informative incentive region onto the feasible action space (Lemma~\ref{lemma:NE_bijection_mapping}). This permits an estimator that is conditioned on observations of informative agent responses.

Second, without knowledge of the informative region, we develop a preference estimator equipped with an indicator function that selects all informative responses. We prove that this estimator is strongly consistent almost surely under the minimal excitation condition~\cite{laiLeastSquaresEstimates1982} on the information matrix. This substantially relaxes the persistence-of-excitation conditions commonly adopted in the incentive design literature and enables active learning and incentive design to be carried out simultaneously with no regret.

Third, we address the dual problem of preference estimation and incentive commitment. To this end, we propose a switching-based incentive design algorithm (Algorithm \ref{alg:one}) that alternates between probing incentives for excitation (exploration) and estimate-based incentives for regret minimization (exploitation). We prove that the resulting policy is asymptotically optimal and achieves an $O(t^{-0.5})$ parameter estimation error and an $O(t^{0.5} \log t)$ squared social-cost regret almost surely.

Fourth, we extend RAID to settings with endogenous perturbations in agent costs. This makes the standard least-squares estimator biased and leads to an Error-in-Variables (EIV) regression problem. We address adaptive incentive design in such scenarios by developing a repeated-sampling estimator together with a switching-based incentive mechanism (Algorithm \ref{alg:two}). When the action space is sufficiently large, the proposed algorithm is shown to retain the same almost-sure estimation and regret rates.

\subsection{Paper Organization and Notations}

Section~\ref{sec:preliminaries} formalizes the incentive design problem and establishes the optimality of affine incentives. Section~\ref{sec:problem_formulation} introduces the agents' response model and formulates the no-Regret Adaptive Incentive Design problem. Section~\ref{sec:measurement_noise} designs a least-squares estimator for the measurement-noise setting, and proposes a switching incentive design algorithm, proving the almost-sure type estimation and regret rates. Section \ref{sec:endogenous_noise} extends RAID to the endogenous-noise setting and introduces a repeated-sampling estimator. The section describes a modified algorithm and derives the corresponding consistency and regret guarantees. Section~\ref{sec:numerical_examples} provides numerical experiments that validate theoretical predictions, and Section~\ref{sec:conclusion} concludes and discusses directions for future work.

\noindent\textbf{Notations.}
Given $n \in \mathbb{N}$, let $[n] = \{1, 2, \ldots, n\}$. For a vector $x \in \mathbb{R}^n$, denote by $\|x\|_2$ and $\|x\|_\infty$ the Euclidean and $\ell_\infty$ norms, respectively. 
For a matrix $A \in \mathbb{R}^{n \times m}$, denote by $\|A\|_2$ the induced operator norm, $\|A\|_\infty$ the max-norm, and $\|A\|_F$ the Frobenius norm. 
For a vector-valued map $f:\mathbb{R}^n \rightarrow \mathbb{R}^m$ with $f(x) = (f_1(x), \, \ldots,\,  f_m(x))^\top$, let ${\nabla} f: {\mathbb{R}}^n \to {\mathbb{R}}^{m\times n}$ denote its Jacobian matrix, i.e., ${\nabla} f(x)= [\frac{\partial f_i}{\partial x_j}]_{ij} $. For a square matrix $A$, define $\operatorname{Sym}(A) = \frac{1}{2}(A + A^\top)$ to be the symmetrization of $A$. 
We denote $f(t) = O(g(t))$ if there exist $c,\, t_0> 0$ such that $|f(t)|\leq c|g(t)|$ for all $t\geq t_0$, and
$f(t) = \Omega(g(t))$ if $|f(t)|\geq c|g(t)|$ for all sufficiently large $t$. Moreover, we write $f(t) = \Theta(g(t))$ if both  bounds hold simultaneously, and $f(t) = o(g(t))$ if $ \lim_{t\to\infty}f(t) /g(t) \to 0$. For a multi-index $a = (a_1, \ldots, a_n) \in \mathbb N^n$ and $x\in {\mathbb{R}}^n$, we denote the index total degree $|a| = \sum_{i=1}^n a_i$ and let $x^a = x_1^{a_1}x_2^{a_2}\ldots x_n^{a_n}$. $C^k(\mathcal{X},\mathcal{Y})$ denotes the space of $k$-times continuously differentiable maps from $\mathcal{X}$ to $\mathcal{Y}$.

\section{Preliminaries}
\label{sec:preliminaries}

\subsection{Incentive Design Problems}
Incentive Design (ID) involves a single \textit{system planner} (principal) and $n$ \textit{agents} (players). 
The agents play a noncooperative game in which each agent $i\in [n]$ minimizes its own cost by choosing a strategy $x_i \in \mathcal X_i$, with $\mathcal{X}_i\subset \mathbb R$ a compact, convex set.~\footnote{For notational simplicity, we assume scalar agent strategies. Results can be extended to arbitrary finite-dimensional strategy spaces.} 
To influence agents' behavior, the system planner declares an incentive law $\gamma_i: \mathcal X \to \mathbb R$ to each agent $i$, where $\mathcal X \!=\!\prod_{i\in[n]}\!\mathcal X_i$. This incentive can be regarded as a reward for the agent when $\gamma_i(x)\leq 0$, and as a penalty (or tax) otherwise. Denote the collective strategy and incentive by $x=(x_i)_{i\in[n]}$ and $\gamma(x)=(\gamma_i(x))_{i\in[n]}$, respectively.

Given an incentive $\gamma: \mathcal X \to \mathbb R^n$, each agent $i$ aims to minimize the cost
\begin{equation}\label{eq:indiv_cost}
c^{\gamma}_i(x) = \ell_i(x) + \gamma_i(x), \qquad x\in\mathcal X, 
\end{equation}
where $\ell_i: \mathcal X \to \mathbb R$ is agent $i$'s \textit{nominal cost}, representing its inherent preference over the strategy $x$.
The additive structure in \eqref{eq:indiv_cost}, whereby an incentive (e.g., payment, wage, or monetary transfer) is added to the nominal cost, is standard in incentive and contract theory\cite{boltonContractTheory2005}.

The next definition characterizes the Nash Equilibrium of the incentivized game played among the agents.
\begin{definition}[Incentivized Nash Equilibrium~-~INE]
Given an incentive law $\gamma(\cdot)$, we say a strategy $x\in \mathcal X$ is an \textit{Incentivized Nash Equilibrium} (INE) if, for each $i\in[n]$, 
\[
c_i^\gamma(x_i, x_{-i})\leq c_i^\gamma(x_i', x_{-i}), \quad \forall  x_i' \in \mathcal X_i,
\]
where $x_{-i}:= (x_j)_{j\in[n]\setminus\{i\}}$.
\end{definition}

On the system planner's side, her objective is to minimize a social cost function $\Psi: \mathcal X \times \mathbb R^n \to \mathbb R$, which is a function of the joint agent strategy and the payment made. When $\Psi$ is Lipschitz, strictly convex, and coercive in its second argument, it has a unique minimizer
\[
(x^{\dagger},\, u^\dagger)=\operatorname*{arg\,min}_{x\in \mathcal X,\, u \in \mathbb R^n} \Psi(x,u).
\]

\begin{assumption}\label{ass:feasibility}
    Social cost $\Psi$ is ${L_\Psi}$-Lipschitz, strictly convex in both arguments, and coercive in its second argument. The desired agent strategy $x^\dagger \in \operatorname{int} \mathcal{X}$.
\end{assumption}

We now state the \textit{Incentive Design }(ID) problem.
\begin{problem}[Incentive Design -  ID]\label{problem:ID}
Given the desired action profile $(x^{\dagger},u^\dagger)$, find an incentive law $\gamma: \mathcal X \to \mathbb R^{n}$ such that $(i)$ $x^\dagger$ is an INE in the induced game with costs~\eqref{eq:indiv_cost}; and $(ii)$ $\gamma(x^\dagger)= u^\dagger$, i.e., the implemented incentive matches.

\end{problem}

In Problem \ref{problem:ID}, the exact social optimum $(x^{\dagger},\, u^\dagger)$ is always achievable, unlike in classical Stackelberg models, because the additive-cost formulation allows for cost reshaping as needed. The incentive design problem was first studied from a control-theoretic perspective in~\cite{hoControltheoreticViewIncentives1980}.

\subsection{Optimality of Affine Incentive Maps}
Linear or affine incentives, as the simplest incentive scheme, are widely used in incentive design (see, e.g., ~\cite{maheshwariAdaptiveIncentiveDesign2024, liSociallyOptimalEnergy2024}). Here we show that, within the class of separable convex functions, affine incentives are indeed optimal for Problem \ref{problem:ID}.
Define the incentive mapping class \(\Gamma =\{\gamma\in \mathcal C^2 ( \mathcal X, \mathbb R^n): \frac{\partial}{\partial x_j}\gamma_i(x)=0, \frac{\partial^2}{\partial x_i^2}\gamma_i(x)\geq 0, \forall i\neq j, \forall x\in \mathcal X\}.\) 
For costs $\ell_i, c_i^\gamma\in \mathcal C^2(\mathcal X, \mathbb R)$,
define the \textit{pseudo-gradient} of the nominal and incentivized games, respectively, as
\begin{equation}
\label{eq:pseudogradient}
G_0(x) = \begin{bmatrix}
    \frac{\partial }{\partial x_1} \ell_1 (x)\\
    \frac{\partial }{\partial x_2} \ell_2 (x)\\
    \vdots\\
    \frac{\partial }{\partial x_n} \ell_n (x)
\end{bmatrix},\
G_\gamma(x) = \begin{bmatrix}
    \frac{\partial }{\partial x_1} c_1^\gamma (x)\\
    \frac{\partial }{\partial x_2} c_2^\gamma (x)\\
    \vdots\\
    \frac{\partial }{\partial x_n} c_n^\gamma (x)
\end{bmatrix}.
\end{equation}
In general, $\nabla G_0(x)$ and $\nabla G_\gamma(x)$ are not symmetric.

\begin{definition}
The nominal costs $\{\ell_i\}_{i\in [n]}$ are diagonally strictly convex if there exists a positive definite diagonal matrix $R \in {\mathbb{R}}^{n\times n}_+$ such that  $\operatorname{Sym}(R \nabla G_0(x)) \succ 0, \forall x \in {\mathcal{X}}$.
\end{definition}
Diagonal strict convexity of the nominal costs is sufficient for existence and uniqueness of the Nash equilibrium in the nominal game~\cite{rosenExistenceUniquenessEquilibrium1965a}. The next lemma then shows that an affine incentive law is optimal for Problem~\ref{problem:ID} in $\Gamma$.

\begin{lemma}
\label{lem:linear_inccentive_is_optimal}
    Suppose that $\ell_i \in \mathcal C^2(\mathcal X, \mathbb R)$, and $\{\ell_i\}_{i\in[n]}$ are diagonally strictly convex. Then the affine incentive law
\begin{equation}\label{eq:optimal_policy_general_0}
\gamma_i^*(x) =p_i^* x_i + (u_i^\dagger- p_i^* x_i^\dagger), \quad i\in[n],
\end{equation}
with $p_i^*:= -\frac{\partial}{\partial x_i} \ell_i(x^\dagger)$, is an optimal solution to Problem \ref{problem:ID}, in the sense that, for any other solution $\gamma \in \Gamma$,
\[
\gamma_i^*(x)\leq \gamma_i(x), \quad  \forall x\in \mathcal X, \, \forall i\in [n],
\]
i.e., $\gamma^*$ minimizes implemented incentives for all agent actions.
\end{lemma}
\begin{proof}
The proof is given in 
Appendix~\ref{app:proofs_auxiliary}.
\end{proof}

Evaluating incentives in $\Gamma$ does not require observing other agents' strategies, which implies the principal provides agents with direct guidance on action. Moreover, convexity of each $\gamma_i$ ensures that  well-posedness of the induced game (existence and uniqueness of the INE) is independent of each particular incentive designed. The optimality of affine incentives has also been studied for the single-follower setting with 
quadratic costs from a robustness perspective \cite{xiaopingliuOptimalIncentiveStrategy1992}.

\section{Problem Formulation}
\label{sec:problem_formulation}

A fundamental challenge in incentive design is the \textit{information asymmetry} that arises when the agents' nominal costs $\ell_i$ are unknown to the planner. It
prevents the planner from directly imposing the optimal incentive $\gamma^*$ given in \eqref{eq:optimal_policy_general_0}. Instead, the planner must learn agents’ preferences from repeated interaction and observation of their responses. 
To this end, the planner models the agents’ marginal nominal costs through the parametrization
\begin{equation}
\label{eq:parametrization}
\frac{\partial \ell_i(x)}{\partial x_i} = \theta_i^{*\top} \Phi(x),
\end{equation}
where $\theta_i^* \in  \mathbb R^D$, called the \textit{type parameter}, is private to agent $i$ and unavailable to the planner. The feature map $\Phi:{\mathcal{X}} \rightarrow \mathbb R^D$ collects all monomials of $n$ variables up to a selected degree $d$ in the form
\begin{equation}
\label{eq:Phi_map}
\Phi(x) = \begin{bmatrix}
\phi_0(x) & \phi_1(x)& \ldots & \phi_d(x)
\end{bmatrix}^\top,
\end{equation}
where $\phi_k(x) = ( x^a)_{a \in \mathcal{A}_k}$, $\mathcal{A}_k = \{a \in \mathbb N^n:\, 
a^T \mathbbm{1} = k\}$. Given the number of agents $n$ and the highest total degree of monomials $d$, the dimension of $\Phi$ is $D = \binom{n+d}{d}$.

\begin{remark}
    Parametrization 
\eqref{eq:parametrization} captures only those terms in $\ell_i(x)$ that depend on $x_i$, since the other terms are strategically irrelevant to agent $i$'s decision-making.
\end{remark}

Let $\Theta^* = (\theta_1^*, \theta_2^*, \ldots, \theta_n^*)^\top \in {\mathbb{R}}^{n \times D}$. We then require the following regularity assumption on the type matrix $\Theta^*$.
\begin{assumption}
\label{ass:types}
There exist a constant $m>0$, and a diagonal matrix
$R \in \mathbb R_+^{n \times n}$ with positive entries such that $\Theta^*\in \mathcal{S}_{\theta}$, where the admissible type set 
\[ \mathcal{S}_{\theta} =\{\Theta \!\in\! {\mathbb{R}}^{n \times D} \!: \frac{1}{2}m\mathbb I \preceq \operatorname{Sym}(R\Theta \nabla\Phi(x)),\forall x \!\in\! \mathcal{X}\}.\]
\end{assumption}

The type set $\mathcal{S}_{\theta}$ is closed and convex.
By Lemma~\ref{lem:linear_inccentive_is_optimal}, the optimal incentive that solves Problem~\ref{problem:ID} admits an affine form.  Therefore, we restrict our attention to incentive laws of the form
\begin{equation}\label{eq:optimal_policy}
\gamma_i(x, p)= p_i x_i + (u_i^\dagger- p_i x_i^\dagger), \quad i\in[n],
\end{equation}
where the collective incentive parameter $p=(p_i)_{i\in[n]} \in \mathbb{R}^n$. In the rest of the paper, whenever no ambiguity arises, we refer to incentive vector $p\in \mathbb R^n$ interchangeably with the corresponding law $\gamma(\cdot) = (\gamma_i(\cdot))_{i\in[n]}$ in \eqref{eq:optimal_policy} parametrized by $p$, as the incentive issued by the planner. Observe that due to~\eqref{eq:optimal_policy}, the pseudo-gradient maps in~\eqref{eq:pseudogradient} satisfy $G_\gamma(x) = G_0(x) + p$. 

\subsection{Response Maps}
We call a map $x^*:\mathbb R^n \rightarrow \mathcal{X}$ a \textit{response map} if, for any $p\in \mathbb R^n$, 
$x^*(p)$ is an INE of the induced game $\{c_i^\gamma(x)\}_i$ with the incentive $\gamma(x,p)$. The following result characterizes the response map $x^*(\cdot)$ under Assumption \ref{ass:types} and shows that its restriction to an open subset of $\mathbb R^n$ is a diffeomorphism onto $\operatorname{int} \mathcal X$.

\begin{lemma}
\label{lemma:NE_bijection_mapping}
    Suppose Assumption \ref{ass:types} holds. Then the response map $x^\ast(\cdot)$ is unique.
      Moreover, on the set $\mathcal{P}=-\Theta^* \Phi(\operatorname{int}{\mathcal{X}})\subset {\mathbb{R}}^n$,
    there exists a diffeomorphism 
    \[H: \mathcal{P} \to \operatorname{int} {\mathcal{X}},
    \]
    such that $H(p)=x^*(p)$ for all $ p\in\mathcal P$. In addition, for every $p\in\mathcal P$ and $w\in\mathbb R^n$,
    \begin{align}
        h_1\|w\|_2 \leq \|\nabla H(p) w\|_2 &\leq h_2\|w\|_2, \text{ and} \label{eq:H_singular_values_bounded} \\
        \Theta^{*} \Phi(H(p)) + p &=0, \label{eq:first_order_condition}
    \end{align}
     where constants $h_1 = (\max_{x\in {\mathcal{X}}}\! \|\nabla G_0(x)\|_2)^{-1}$, and $h_2 = 2\lambda_{\max}(R)/{m}$.
\end{lemma}

\begin{proof}
Presented in Appendix~\ref{app:proofs_auxiliary}.
\end{proof}

The set $\mathcal{P}=-\Theta^* \Phi(\operatorname{int}{\mathcal{X}})$ is bounded and simply connected (hence path connected), since $H: \mathcal{P} \to \operatorname{int} {\mathcal{X}}$ is a diffeomorphism and $\operatorname{int} {\mathcal{X}}$ 
is an open hyperrectangle in $\mathbb R^n$. However, $\mathcal{P}$ need not be convex. 

\begin{figure*}[b!]
\normalsize
\vspace*{4pt}
\setcounter{MYtempeqncnt}{\value{equation}}
\setcounter{equation}{13}
\hrulefill
\begin{subequations}
\label{eq:solution_estimator}
\begin{equation}
\label{eq:theta_solution_estimator}
\begin{aligned}
    \hat \Theta (t)  =
    \hat \Theta(t-1)  - \delta(t) (\hat p(t) + \hat\Theta(t-1)\xi(t)) \frac{\xi(t)^\top S_{t-1}}{1+ \delta(t)\xi(t)^\top S_{t-1}\xi(t) },
\end{aligned}
\end{equation}
\begin{equation}
\label{eq:Sigma_solution_estimator}
\begin{aligned}
        S_t = S_{t-1} - \delta(t)S_{t-1}\xi(t)
        \frac{\xi(t)^\top S_{t-1}}
      {1 + \delta(t)\,\xi(t)^\top S_{t-1}\xi(t)}.
\end{aligned}
\end{equation} 
\end{subequations}
\setcounter{equation}{\value{MYtempeqncnt}}
\stepcounter{equation}
\end{figure*}

\begin{remark}
 A related result to Lemma~\ref{lemma:NE_bijection_mapping}, in the setting where each $\ell_i$ depends only on $x_i$, was reported in \cite{vasileiou2026adaptiveincentivedesignregret}. Here, our extension relies on the strong monotonicity assumption (Assumption~\ref{ass:types}) on the game's pseudo-gradient.
\end{remark}

\subsection{Two Noise Cases}
The planner can learn the unknown costs $\ell_i$, equivalently, the parameter $\Theta^*$,  only from NE responses that vary smoothly with incentives. We refer to the set $\mathcal{P}$ in Lemma~\ref{lemma:NE_bijection_mapping} as the \textit{informative region}, and 
call any $p \in \mathcal{P}$ an \textit{informative incentive}. 
Every informative incentive $p \in \mathcal{P}$ elicits a unique response $x^*(p) \in \operatorname{int} \mathcal{X}$ satisfying
\begin{equation}\label{eq:observation_no_noise}
    \Theta^*\Phi(x^*(p)) + p=0, \quad \forall p\in \mathcal{P}. 
\end{equation}
\begin{remark}
    The informative region $\mathcal{P}$ depends on the agents' types $\Theta^*$, as evidenced by \eqref{eq:first_order_condition}, and is therefore unknown to the system planner.
\end{remark}

Without prior information of parameter $\Theta^*$, the planner can jointly design an incentive sequence $\{p(t)\}_{t\geq 1}$, apply the corresponding incentive $\gamma(x,p(t))$, and estimate $\Theta^*$ based on observed agent responses. Let the incentive–response trajectory available to the planner (or an estimator) up to time $t\in\mathbb{Z}_+$ be $\{(\hat p(\tau), \hat x(\tau))\}_{\tau=1}^t$.
We consider a noisy extension of the ideal best-response model in \eqref{eq:observation_no_noise}, as
\begin{equation}
\begin{aligned}
\label{eq:model}
p(t) + v(t) & = - \Theta^* \Phi(\hat x(t)) , \\
\hat p(t) & = p(t) + e(t),
\end{aligned}
\end{equation}
where $v(t)\in \mathbb R^n$ is an endogenous system noise, $e(t)\in \mathbb R^n$ is an exogenous measurement noise, $\hat x(t)$ represents the observed agents' response, and $p(t)$ and $\hat p(t)$ are the incentive issued by system planner and observed by an exogenous estimator, respectively. 
In this paper, we investigate two scenarios of the noise processes:
\begin{enumerate}
    \item[(i)] \textit{Measurement-noise case}: $v(t)=0$ for all $t\geq 1$, i.e., endogenous noise is absent.
    \item[(ii)] \textit{Endogenous-noise case}: $v(t)\neq 0$ for some $t\geq 1$.
\end{enumerate}
In particular, case (ii) is more general and captures the setting where response $\hat x(t)$ is correlated with the noise process. Parameter estimation in this case then corresponds to an Error-in-Variables (EIV) regression problem. We will investigate parameter estimation and incentive design algorithms for the measurement-noise case in Section~\ref{sec:measurement_noise}, and for the endogenous-noise case in Section~\ref{sec:endogenous_noise}, respectively.

\begin{remark}
The measurement-noise case of adaptive incentive design was first proposed in \cite{ratliffAdaptiveIncentiveDesign2021}, reflecting a setting where parameter (type) estimation is performed by a distinct entity, such as a statistical agency. Regret minimization in this case with uncoupled agents was investigated in~\cite{vasileiou2026adaptiveincentivedesignregret}.
\end{remark}

\subsection{No-Regret Adaptive Incentive Design}
In this paper, we quantify the performance of an incentive design algorithm in terms of the cumulative divergence from the social optimum. In this setting, 
incentive $p^* = - \Theta^* \Phi(x^\dagger)$ defines the affine law 
$\gamma(\cdot, p^*)$ that elicits the desired response $x^\dagger$ from the agents and simultaneously satisfies $u^\dagger = \gamma(x^\dagger, p^*)$.  Given an incentive sequence $\{p(\tau)\}_{\tau \geq 1}$, we define the planner's $t$-stage squared social-cost regret as
\begin{equation}
\label{eq:regret_definition}
\begin{aligned}
       \mathcal{R}_t & = \sum_{\tau = 1}^t\left(\Psi(\hat x(\tau), u(\tau)) - \Psi(x^\dagger, u^\dagger)\right)^2,
\end{aligned}
\end{equation}
where $\Psi(\cdot,\cdot)$ is the social cost function, and $u(\tau)  = \gamma(\hat{x}(\tau), p(\tau))$ is the implemented incentive.

The planner's objective is therefore twofold: to ensure accurate type estimation for the unknown parameter $\Theta^\star$ through sufficient exploration of the parameter space and regulate the collective behavior toward the social optimum with minimal regret. These objectives are formalized in the problem below.

\begin{problem}[No-Regret
Adaptive Incentive Design - RAID]
    Given a social cost $\Psi$ and a desired action profile $(x^\dagger, u^\dagger)\in \operatorname*{arg\,min}_{x\in \mathcal{X}, \, u\in \mathbb R^n} \Psi(x, u)$,
    a system planner with no prior knowledge of 
    $\Theta^* = (\theta_1^*, \ldots, \theta_n^*)^\top$
    aims to design an estimator 
    $\hat \Theta(t) =\mathcal{F}_t(\hat p(t), \hat x(t), \hat \Theta(t-1) ) $ 
    and an incentive law $p(t)= \mathcal{G}_t (\hat \Theta(t-1), x^\dagger, u^\dagger)$  
    satisfying:
    \begin{enumerate}
        \item \emph{Strong consistency} of $ \hat \Theta(t)$: For any initial estimate $\hat \Theta(0)$, $\lim_{t\to \infty} \hat \Theta(t)= \Theta^*$ almost surely (a.s.), and
        \item \emph{Sublinear regret}: The $t$-stage average regret $t^{-1}\mathcal{R}_t$ of policy  
        $\{p(\tau)\}_{\tau =1}^t$
        vanishes asymptotically a.s., that is,
        \(\mathcal{R}_t = o(t) \text{ a.s. }\)
    \end{enumerate}
\end{problem}

Solving the RAID problem inherently involves an exploration–exploitation trade-off. The planner seeks to exploit the current parameter estimates to induce desired agent responses and simultaneously, actively explore the parameter space to ensure statistical identifiability and consistency of $\Theta^*$. The next section develops type estimation and incentive design algorithms with this trade-off for the measurement-noise case.

\section{No-Regret Adaptive Incentive Design:\\The Measurement-Noise Case}
\label{sec:measurement_noise}

This section addresses the RAID problem when the endogenous process is absent, that is $v(t)=0$ for all $t\geq 1$. We first construct a recursive type estimator from informative incentive-response observations and then characterize its strong consistency in terms of the information accumulated along the observed trajectory.

\begin{assumption}[Measurement-noise case]
\label{ass:independet_noise}
Assume the process $\{v(t)\}_{t}$ is identically zero. 
Moreover, the process $\{e(t)\}_{t}$ is i.i.d., zero-mean and satisfies $\mathbb E (e(t)e(t)^\top) = \sigma_e^2 \mathbb I$ and $\sup_t\mathbb E\|e(t)\|_\infty^a < \infty$ for some $a > 2$, and $\sigma_e>0$.
\end{assumption}

\subsection{Type Estimation and Strong Consistency}
\label{sec:type_estimation}

Given the collection of observations $\{(\hat p(\tau), \hat x(\tau))\}_{\tau=1}^t$,
the natural least-squares estimator $\hat \Theta(t)$ of $\Theta^*$ is
\begin{equation}
\label{eq:theta_minimization_problem}
    \hat \Theta(t)  = \operatorname*{arg\,min}_{\Theta \in {\mathbb{R}}^{n\times D}} \sum_{\substack{\tau \leq t: \\ 
    p(\tau) \in \mathcal{P}}} 
    \big\|\Theta \Phi(\hat x(\tau)) + \hat p(\tau)\big\|^2_2.  
\end{equation}
The summation in~\eqref{eq:theta_minimization_problem} is restricted to informative time indices, since the response model~\eqref{eq:model} holds only when the issued incentive $p(t)$ lies in the informative region $\mathcal P$.
Although $\mathcal{P}$ depends on the unknown type matrix $\Theta^*$, Lemma \ref{lemma:NE_bijection_mapping} implies that the event $\{p(t) \in \mathcal{P}\}$ is equivalent to $\{ \hat x(t) = x^*(p(t))\in \operatorname{int} \mathcal{X}\}$. Thus informativeness can be verified directly from the observed response, and is not influenced by the measurement-noise process .
Minimization \eqref{eq:theta_minimization_problem} can be solved by the Recursive Least Squares (RLS) estimator \eqref{eq:solution_estimator}, 
where $\xi(t) = \Phi(\hat x(t))$, $\delta(t) = \mathbf{1} \{\hat x(t) \in \operatorname{int} \mathcal{X}\}$, and $\mathbf{1}\{\cdot\}$ is the indicator function. 
By the Sherman-Morrison formula, \eqref{eq:Sigma_solution_estimator} represents a rank-one update on the matrix $S_{t}^{-1}$, that is,
\begin{equation}
\label{eq:Sigma_inverse_rank_one_update}
    S_t^{-1} = S_{t-1}^{-1} + \delta(t)\xi(t)\xi(t)^\top. \tag{\ref{eq:Sigma_solution_estimator}'}
\end{equation}
The Gram matrix $S_{t}^{-1}$ is the information matrix associated with problem \eqref{eq:theta_minimization_problem}. We denote its  smallest and largest eigenvalues by $\lambda_{\min}(t)$ and $\lambda_{\max}(t)$, respectively.

\setcounter{equation}{14}

We next give a sufficient condition for strong consistency of \eqref{eq:solution_estimator} with respect to the growth of $\lambda_{\min}(t)$.

\begin{lemma}[Developed from~\cite{laiLeastSquaresEstimates1982}]
\label{thm:sufficient_strong_consistency}
    Suppose Assumptions \ref{ass:types} and \ref{ass:independet_noise} hold. 
    If $ \log t = o (\lambda_{\min}(t))$, then
    the estimator $\hat \Theta(t)$ in~\eqref{eq:solution_estimator} satisfies
    \[\|\hat \Theta(t) - \Theta^*\|^2_F = O\Big(\frac{ \log t}{\lambda_{\min}(t)}\Big) \text{ a.s.},\]
    where $\Theta^*$ is the true type matrix.
    Consequently, if $ \log t = o (\lambda_{\min}(t))$ a.s., then $ \hat \Theta(t) \rightarrow\Theta^*$ a.s.
\end{lemma}
\begin{proof}
Presented in Appendix \ref{app:proofs_main}.
\end{proof}

In Lemma~\ref{thm:sufficient_strong_consistency}, $\lambda_{\min}(t)$ quantifies the information accumulated by the estimator.
The growth condition $\log t = o (\lambda_{\min}(t))$ is a diminishing excitation requirement, which is substantially milder than the standard persistence-of-excitation conditions commonly imposed in adaptive ID literature. 

The next result establishes that i.i.d.\ Gaussian probing incentives provide sufficient excitation despite the nonlinear transfer map $\Phi\circ H(\cdot)$ and the unknown informative region. In particular, even after observations are filtered through $\delta(t)$, the resulting regressors yield linear growth of $\lambda_{\min}(t)$ and therefore satisfy the requirement of Lemma~\ref{thm:sufficient_strong_consistency}.

\begin{theorem}
\label{thm:information_growth}
    Suppose Assumptions \ref{ass:types} and \ref{ass:independet_noise} hold, and let $\{p(\tau)\}_{\tau \leq t} \subset {\mathbb{R}}^n$ be 
    an incentive sequence consisting of i.i.d.\ incentives 
    $p(t)\sim \mathcal{N}(0,\sigma^2_p\ \mathbb I)$.
    Then
    \(\lambda_{\min}(t) = \Theta(t) \text{ a.s.}\)
    
\end{theorem}
\begin{proof}
    Presented in 
    Appendix \ref{app:proofs_main}.
\end{proof}

An intermediate result, which is also of independent interest, is that i.i.d.\ Gaussian incentives are persistently exciting for the regression \eqref{eq:model}, 
despite individual probing incentives not necessarily belonging to region \(\mathcal P\).

\begin{lemma}
\label{lemma:excitation}
    Let Assumptions \ref{ass:types} and \ref{ass:independet_noise} hold, and let $\{p(\tau)\}_{\tau =1}^t$ be an ${\mathbb{R}}^n$-valued i.i.d. incentive sequence with $p(t)\sim \mathcal{N}(0,\sigma^2_p\ \mathbb I)$.
    Then there exists a constant $c>0$ such that 
    \[  \mathbb E\left( \xi(t)\xi(t)^\top \delta(t)\right) \succeq c\mathbb I,\]
    where $\xi(t)$ and  $\delta(t)$ are defined in \eqref{eq:solution_estimator}.
\end{lemma}
\begin{proof}
    Presented in Appendix \ref{app:proofs_auxiliary}.
\end{proof}

The linear growth established in Theorem~\ref{thm:information_growth} is stronger than the excitation condition required in Lemma~\ref{thm:sufficient_strong_consistency}. This gap leaves room for the no-regret incentive design developed next, in which only intermittent probing is needed to maintain consistency of the type estimate.

\subsection{RAID Algorithm and Regret Guarantees}
\label{sec:algorithm}
Next, we leverage 
the estimator in Section \ref{sec:type_estimation} to design an incentive algorithm that simultaneously achieves both objectives of the no-Regret Adaptive Incentive Design (RAID) problem in the measurement-noise setting. The resulting mechanism is summarized in Algorithm \ref{alg:one} and alternates between probing (exploration) and estimate-based (exploitation) incentives.

\begin{algorithm}
\caption{no--Regret Adaptive Incentive Design (RAID) in the Measurement-Noise Case}\label{alg:one}
\begin{algorithmic}[1]
    \Statex \textbf{Parameters:}
    \Statex Parameters $\gamma \in [\frac{1}{2}, 1)$, $\sigma^2_p > 0$
    \Statex Desired profile $(x^\dagger, u^\dagger)$ with $x^\dagger\in \operatorname{int} \mathcal{X}$
    \Statex Iterations $t\geq 1$.
    \Statex \textbf{Output:}
    \Statex Estimate and incentive $\{ \hat\Theta(\tau), p(\tau)\}_{\tau=1}^t$
    \Statex
    \State Define function $A(\tau) = \tau^{\gamma}\log \tau$
    \State Initialize $S_1 \succ 0$, $p(1)$, and $\hat \Theta(1)$
    \NoIndentFor{$\tau = 2, \ldots, t$}
    \State $\textrm{Explore}_\tau \gets \big( \operatorname{tr}(S_{\tau-1}) > A(\tau-1)^{-1}\big)$
    \If{$\textrm{Explore}_\tau$} \Comment{Exploration phase}
    \State Draw incentive parameter $p(\tau) \sim \mathcal N(0,\sigma^2_p\ \mathbb I)$
    \State Issue incentive $\gamma(\cdot, p(\tau))$
    \State Observe agent response $\hat x(\tau) = x^*(p(\tau))$
    \State Update $S_{\tau}$ and $\hat \Theta(\tau)$, according to \eqref{eq:solution_estimator}
    \Else \Comment{Exploitation phase}
        \State Update $p(\tau) \gets  -\hat\Theta(\tau-1)  \Phi(x^\dagger)$
        \State Issue incentive $\gamma(\cdot,p(\tau))$
        \State Maintain $S_{\tau} \gets S_{\tau-1}$ and $\hat \Theta(\tau) \gets \hat \Theta(\tau-1)$
    \EndIf
    \EndNoIndentFor
\end{algorithmic}
\end{algorithm}

As illustrated by Lemma~\ref{lem:linear_inccentive_is_optimal}, the planner's optimal incentive is $\gamma^*(\cdot, p^*)$ with parameter $p^* = - \Theta^* \Phi(x^\dagger)$.
Without prior knowledge of agent types, the system planner naturally substitutes the current estimate $\hat \Theta(t)$ for the true $\Theta^*$, which defines the policy $p(\tau) = -\hat\Theta(\tau-1)  \Phi(x^\dagger)$.

By Lemma~\ref{thm:sufficient_strong_consistency}, strong consistency of the estimator $\hat \Theta(t)$ requires the condition $ \log t = o(\lambda_{\min}(t))$ a.s.
To this end, Algorithm~\ref{alg:one} uses a threshold-based switching rule to alternate between exploration and exploitation phases. In the remainder of the section, we consider $t$ large enough to allow asymptotic analysis. Let
\begin{equation}
\label{eq:control}
    p(t) = \begin{cases}
        -\hat \Theta(t-1) \Phi(x^\dagger), & t\in[\tau_k, \sigma_k) \\
        \varepsilon(t), & t\in[\sigma_k, \tau_{k+1})
    \end{cases}
\end{equation}
where the probing sequence $\{\varepsilon(t)\}_{t}$ is i.i.d. with each $\varepsilon(t) \sim \mathcal{N}(0, \sigma^2_p\, \mathbb I)$, and the switching schedule $\{\tau_k\}_{k}$ and $\{\sigma_k\}_{k}$ is designed iteratively, according to the threshold
$\{A(t)\}_{t}$. Specifically, let $\tau_1 = 0$, and
\begin{equation}
\label{eq:switches}
\begin{aligned}
    \sigma_k & = \inf\{ t\geq \tau_k : \operatorname{tr}(S_{t}) > A(t)^{-1}\}, \\
    \tau_{k+1} & = \inf\{ t\geq \sigma_k : \operatorname{tr}(S_{t}) \leq A(t)^{-1}\}, \\
\end{aligned}
\end{equation}
for $k\geq 1$, where $A(t) = t^\gamma \log(t)$, as in Algorithm \ref{alg:one}.
 
In the proposed algorithm, the \emph{exploitation phase} occurs at iterations $t \in [\tau_k, \sigma_k)$, $k\geq 1$, and utilizes $\hat \Theta(t)$ to regulate agents' behavior 
to the target $x^\dagger$. During exploitation phases, the planner does not update the type estimate. On the other hand, \emph{exploration phases} occur when the information matrix $S_{t}^{-1}$ needs additional excitation. 
Notice that $1/\lambda_{\min}(S_{t}^{-1}) \leq \operatorname{tr}(S_{t}) \leq D/ \lambda_{\min}(S_{t}^{-1})$, and therefore, designing the switching condition according to $\operatorname{tr}(S_{t})$ yields a more computationally tractable switching signal. In particular, $\lambda_{\min}(t) \leq A(t)$ implies that $\operatorname{tr}(S_{t}) \geq A(t)^{-1}$. This condition is used in \eqref{eq:switches} to identify iterations at which excitation is insufficient.
At such times, probing incentives are introduced and \eqref{eq:solution_estimator} is used to update $\hat \Theta(t)$.

Applying Algorithm \ref{alg:one} with incentive policy \eqref{eq:control} and switching schedule \eqref{eq:switches}, the planner obtains
an almost-sure convergence rate of $\hat \Theta(t)$ to $\Theta^*$ and guarantees that the $t$-stage average regret $t^{-1}\mathcal{R}_t$ vanishes almost surely. This constitutes the first main result of this work, and is presented below.

\begin{theorem}
\label{thm:algorithm_regret}
    Let Assumptions \ref{ass:feasibility}, \ref{ass:types} and \ref{ass:independet_noise} hold. For all $\gamma \in [\frac{1}{2},1)$, the type estimates $\{\hat \Theta(\tau)\}_{\tau =1}^t$ produced by Algorithm \ref{alg:one} satisfy
    \[\| \hat \Theta(t) - \Theta^*\|_F^2 = O(t^{-\gamma}) \text{ a.s.}\]
    Moreover, the system planner's regret $\mathcal{R}_t$, as defined in \eqref{eq:regret_definition}, satisfies
    \[ \mathcal{R}_t = O (t^\gamma \log t)
    \text{ a.s.}\]
\end{theorem}
\begin{proof}
    Presented in Appendix \ref{app:proofs_main}.
\end{proof}

The regret bound of Theorem~\ref{thm:algorithm_regret} is sublinear for all $\gamma \in [1/2,\, 1)$, hence the policy is asymptotically optimal. Among admissible tuning parameter choices, $\gamma=1/2$ yields the fastest decay of $t^{-1} \mathcal{R}_t$ and thus represents the planner’s preferred parameter, at the cost of a slower estimation error decay. 

An intermediate lemma, related to Theorem \ref{thm:algorithm_regret}, is the fact that Algorithm \ref{alg:one} enforces the $\log t = o(\lambda_{\min}(t))$ growth of the information matrix, as is required by Lemma~\ref{thm:sufficient_strong_consistency} for strong consistency of the estimator.
Let $\#(t)$ denote the total number of probing iterations completed by Algorithm \ref{alg:one} up to iteration $t$. The lemma establishes that both probing iterations and the spectrum of the information matrix grow in the same order as the threshold sequence $A(t)$ under Algorithm \ref{alg:one}.

\begin{lemma}
\label{lemma:number_of_exploration_samples}
Let Assumptions \ref{ass:types} and \ref{ass:independet_noise} hold. Under Algorithm \ref{alg:one}, it holds that 
\[\#(t) = \Theta(A(t)),\text{ and }\lambda_{\min}(t) = \Theta(A(t)),\text{ a.s.} \]
\end{lemma}
\begin{proof}
    Presented in Appendix \ref{app:proofs_main}.
\end{proof}

\section{No-Regret Adaptive Incentive Design:\\The Endogenous-Noise Case}
\label{sec:endogenous_noise}

In this section, we investigate the more general endogenous-noise case when the noise $v(t)$ is nonzero in \eqref{eq:model}. This leads to an error-in-variables (EIV) regression problem and a corresponding adaptive incentive design algorithm. Observe that the least-squares estimator~\eqref{eq:solution_estimator} derived in Section~\ref{sec:measurement_noise} can be written as 
\begin{equation*}
\label{eq:ls_estimator}
\hat{\Theta}(t)   = \Theta^{*}\! -\!\Big(\!\sum_{\tau=1}^t\! \delta(\tau) e(\tau) \xi(\tau)^\top\Big)\!\Big( \!\sum_{\tau=1}^t \!\delta(\tau) \xi(\tau) \xi(\tau)^\top\Big)^{-1}\!\!,
\end{equation*}
so strong consistency requires $\mathbb E ( \delta(t)e(t) \Phi(\hat x(t))^\top) = 0$. Under Assumption~\ref{ass:independet_noise}, the measurement noise $e(t)$ is independent of the agents' response $\hat x(t)$, and the condition is satisfied. However, the independence of agent responses and noise processes fails in the endogenous-noise case.

The endogenous noise $v(t)$ can be viewed as model uncertainty. Two representative sources of such uncertainty are $(i)$ misspecification of agents’ cost functions, which induces a discrepancy between the parametrized and actual best-response map; and $(ii)$ time variation in the dependence of $\ell_i$ on $x_i$, for example through a time-varying slope in the term $\partial \ell_i(x)/\partial x_i$. In both cases, the resulting deviation is endogenous to the agent response and cannot be treated as independent measurement noise. We refer to these effects collectively as model uncertainty.

\subsection{Error-in-Variables Analysis}
\label{sec:EIV}
With endogenous noise $\{v(t)\}_t$, agents respond to the effective incentive $ p(t)+v(t)$. Whenever $p(t)+v(t)\in\mathcal P$, the agents' observed response $\hat{x}(t)\in \operatorname{int} {\mathcal{X}}$ and satisfies
\(
\hat{x}(t) = H\bigl(p(t) + v(t)\bigr),
\)
where $H$ is the diffeomorphism in 
Lemma~\ref{lemma:NE_bijection_mapping}. 
Because $\hat x(t)$ depends on $v(t)$, the regressor 
$\xi(t) = \Phi(\hat{x}(t))$ is generally correlated with the noise, resulting in an error-in-variables bias, proportional to the correlation $\mathbb E (v(t) \Phi(\hat x(t))^\top) \neq 0$.
Consequently, the least-squares estimator becomes asymptotically biased and fails to recover $\Theta^*$.

To address the error-in-variables bias, we adopt a repeated-sampling estimator, 
motivated by instrumental variable methods~\cite{hausmanIdentificationEstimationPolynomial1991, liNonparametricEstimationMeasurement1998}. Specifically, at (adaptively) selected times $t$, the system planner fixes $p(t)$ and issues the associated incentive $\gamma(\cdot, p(t))$ three times, obtaining three noisy response observations. We denote the resulting observations by the pairs $(\hat p^{k}(t), \hat x^{k}(t))$, $k\in [3]$. When $\hat x^{k}(t) \in {\mathcal{X}}$, the observation pair satisfies
\begin{equation}
\label{eq:model_eiv_2}
\begin{aligned}
\hat p^{k}(t) & = - \Theta^* \Phi( \hat x^{k}(t)) + e^{k}(t) - v^{k}(t) \\
& = - \Theta^* \xi^{k}(t) + e^{k}(t) - v^{k}(t),
\end{aligned}
\end{equation}
where $\xi^{k}(t) = \Phi(\hat x^{k}(t))$.

To formalize the endogenous-noise setting we substitute Assumption \ref{ass:independet_noise} with the following. 
\begin{assumption}[Endogenous-Noise Case]
\label{ass:eiv_noise}
The noise processes $\{e^k(t)\}_{t}$ and $\{v^k(t)\}_{t}$ in~\eqref{eq:model_eiv_2} are independent across both $k$ and $t$. Moreover, there exist constants $a > 2$ and $\sigma_e, \bar v>0$ such that for each $k\in [3]$, the processes are i.i.d., zero-mean and satisfy 
\[\begin{aligned}
    & \mathbb E(e^k(t)e^k(t)^\top) = \sigma_e^2 \mathbb I, \ \sup_t\mathbb E\|e^k(t)\|_\infty^a < \infty, \\
    & \sup_t\|v^{k}(t)\|_\infty \leq \bar v \text{ a.s.}
\end{aligned}\]
\end{assumption}
Here we require the endogenous noise $v^{k}(t)$ to be finitely supported, and further require ${\mathcal{X}}$ to be large enough. Define the distance-to-boundary and diameter functions
\[d_{{\mathcal{X}}}(x) = \min_{y \in \partial {\mathcal{X}}}\|x-y\|_2, \text{ and } |{\mathcal{X}}|= \max_{x,y\in {\mathcal{X}}}\|x-y\|_2.\]
\begin{assumption}[$\eta$-exploration margin]
\label{ass:exploration_margin}
The feasible response set ${\mathcal{X}}$ is large in the sense that, there exists an $\eta > {4\sqrt{n}\lambda_{\max}(R) \bar v}/{m}$ such that $|{\mathcal{X}}| > 2 \eta$.
\end{assumption}

We also introduce a modified regret $\mathcal{J}_t$ to quantify the performance of a given policy  $\{p(\tau)\}_{\tau=1}^t$. Define
\begin{equation}
\label{eq:eiv_regret}
\begin{aligned}
    &\mathcal{J}_t  = \sum_{\tau=1}^t \big|\Psi\big(\hat x^{1}(\tau), u^{1}(\tau)\big) - \Psi\big(x_{r}^{1}(\tau), u_{r}^{1}(\tau)\big)\big|^2, \\
    & u^{1}(\tau) = \gamma\big(\hat x^{1}(\tau), p(\tau)\big),
\end{aligned}
\end{equation}
where $x_r^{k}(t)$, $u_r^{k}(t)$, denote the noise-adjusted reference sequences $x_{r}^{k} (t) = 
x^*(p^* + v^{k}(t))$ and $u_r^{k}(t) = \gamma(x_{r}^{k} (t), p^*)$, and $p^* = -\Theta^* \Phi(x^\dagger)$.
First, observe that the planner accumulates regret only for one of the independent trials $(\hat p^{k}(t), \hat x^{k}(t))$, $k\in[3]$. This simplifies the exposition between those iterations that the algorithm issues one or three trials of $\gamma(\cdot, p(t))$. It introduces at most a constant multiplicative factor in the regret, so asymptotic results hold without modification. Second, by using reference $x_{r}^{k}(t)$, we  isolate the incremental error introduced by the planner’s decision
to use incentive $p(t)$ instead of the optimal incentive $p^*$. Clearly, the regret $\mathcal J_t$ of policy $\{p(\tau)\}_{\tau=1}^t$ with $p(\tau)=p^*$ at every $\tau$ is zero for any realization of $\{v^{k}(\tau)\}_{\tau=1}^t$, so $x_{r}^{k}(t)$ is the optimal performance obtainable by an oracle in the presence of endogenous-noise.

With these modifications, the RAID problem remains to design an incentive policy that ensures strong consistency of the type estimator under Assumptions~\ref{ass:eiv_noise} and~\ref{ass:exploration_margin}, while guaranteeing a vanishing regret $t^{-1} \mathcal{J}_t$ a.s. 

\subsection{A Repeated-Sampling Estimator and Algorithm}
\label{sec:EIV_estimator}
We construct a modified least-squares estimator that resolves the EIV bias identified in Section~\ref{sec:EIV} and extends the RAID Algorithm to the endogenous-noise setting.
The proposed estimator leverages the conditional independence of $\hat x^{k}(t)$ given the incentive $p(t)$. This eliminates the correlation term without requiring knowledge of $\sigma_e^2$ and $\bar v$ in Assumption~\ref{ass:eiv_noise}. 
The algorithm is presented in Algorithm~\ref{alg:two} and preserves the almost-sure parameter convergence and regret rates as in the measurement-noise case, 
demonstrating that the endogenous-noise extension incurs no asymptotic cost.

Given the observations $\{(\hat p^{k}(\tau), \hat x^{k}(\tau))\}_{\tau=1}^t$, $k\in [3]$, the planner constructs an estimator $\tilde \Theta(t)$ for $\Theta^*$ as 
\begin{subequations}
\label{eq:eiv_estimator}
\begin{align}
\tilde \Theta(t) & = - X_tU_t^+, \\
X_t & =  \sum_{\tau=1}^t \delta(\tau) \hat p^{1}(\tau) \xi^{2}(\tau)^\top, \\
U_t & = \sum_{\tau=1}^t\! \delta(\tau) \xi^{1}(\tau) \xi^{2}(\tau)^{\top},
\end{align}
\end{subequations}
where $\delta(t) =  \mathbf{1} \{ d_{{\mathcal{X}}}(\hat x^{{3}}(t)) \geq \eta \}$, $\eta$ is defined in Assumption~\ref{ass:exploration_margin}, $\xi^{k}(t) = \Phi(\hat x^{k}(t))$,  
and $A^+$ denotes the pseudoinverse. Note that~\eqref{eq:eiv_estimator} only requires the three response observations $\{x^{k}(t)\}_{k\in [3]}$ and one incentive parameter observation $\hat p^{1}(t)$. 
We observe that, when $\delta(t) =1$, both $\hat x^{2}(t), \hat x^{1}(t) \in {\mathcal{X}}$, since
\[\begin{aligned}
    d_{\mathcal{X}}(\hat x^{k}(t)) & \geq d_{\mathcal{X}}(\hat x^{3}(t)) - \|\hat x^{k}(t) - \hat x^{3}(t)\|_2 \\
    & \geq \eta- \|H(p(t)+v^{k}(t))- H(p(t)+v^{3}(t))\|_2\\
    & \ge \eta- 2 h_2 \sqrt{n}\bar v > 0.   
\end{aligned}\]
The following theorem characterizes the strong consistency of estimator $\tilde{\Theta} (t)$ with respect to an information matrix.
\begin{theorem}
\label{thm:eiv_sufficient_strong_consistency}
Suppose Assumptions \ref{ass:types}, \ref{ass:eiv_noise}, and~\ref{ass:exploration_margin} hold. Let
\begin{align}
    B_t & = \sum_{\tau=1}^t \delta(\tau) \xi^{2}(\tau)\xi^{{2}}(\tau)^\top, \label{eq:B_t}
\end{align}
where $\xi^{k}(t)$ and $\delta(t)$ are defined in \eqref{eq:eiv_estimator}. If $\log t = o(\lambda_{\min}(B_t))$, then the estimator $\Tilde{\Theta}(t)$ in \eqref{eq:eiv_estimator} satisfies
\[\|\tilde \Theta(t) - \Theta^*\|_F^2 = O \left( \frac{\log t}{\lambda_{\min}^+(B_t^{-\frac{1}{2}} U_t^\top U_t B_t^{-\frac{1}{2}})}\right) \text{ a.s.,}\]
where $\lambda_{\min}^+(\cdot)$ denotes the smallest nonzero eigenvalue of a matrix. Consequently, if $\log t = o (\lambda_{\min}(B_t^{-\frac{1}{2}} U_t^\top U_t B_t^{-\frac{1}{2}})) $ a.s., then $\tilde \Theta(t) \to \Theta^*$ a.s.
\end{theorem}
\begin{proof}
    Presented in Appendix \ref{app:proofs_main}.
\end{proof}

Matrix $B_t^{-\frac{1}{2}} U_t^\top U_t B_t^{-\frac{1}{2}}$ can be viewed as the endogenous-noise representation of the information matrix $S_{t}^{-1}$  in Section~\ref{sec:measurement_noise}. It is worth noting that Theorem \ref{thm:eiv_sufficient_strong_consistency} requires two separate growth conditions, one on the spectrum of $B_t$, and one on $B_t^{-\frac{1}{2}} U_t^\top U_t B_t^{-\frac{1}{2}}$. The lemma that follows 
proves that normally distributed probing incentives are adequate to ensure both growth conditions. 

\begin{lemma}
\label{lemma:eiv_excitation}
Suppose Assumptions \ref{ass:types}, \ref{ass:eiv_noise}, and~\ref{ass:exploration_margin} hold. Let $\{p(\tau)\}_{\tau=1}^t$ 
be an incentive policy of ${\mathbb{R}}^n$-valued, i.i.d. random variables with $p(t)\sim \mathcal{N}(0, \sigma_p^2\ \mathbb I)$. Then, for $i,j \in \{1,2\}$, there exists a $c_{ij}>0$ such that
\[
\mathbb E[\xi^i(t)\xi^j(t)^\top\delta(t)]
\succeq c_{ij} \mathbb{I},
\]
where $\xi^{k}(t)$, $\delta(t)$ are defined in \eqref{eq:eiv_estimator}.
\end{lemma}
\begin{proof}
    Presented in Appendix \ref{app:proofs_auxiliary}.
\end{proof}

Taken along with the strong law of large numbers, the above lemma implies that both $B_t$ and $U_t$ (consequently, $B_t^{-\frac{1}{2}} U_t^\top U_t B_t^{-\frac{1}{2}}$ as well) grow linearly in $t$ under normally distributed probing incentives, for almost all trajectories of $\{v^{k}(t)\}_t$ and $\{e^{k}(t)\}_t$. Then it is natural to modify Algorithm~\ref{alg:one} so that it maintains the adequate growth of $\lambda_{\min}(B_t^{-\frac{1}{2}} U_t^\top U_t B_t^{-\frac{1}{2}})$. This is achieved via the incentives $p(t)$ in \eqref{eq:control} with $\hat\Theta(t)$ replaced by $\tilde{\Theta}(t)$, and a modified switching schedule $\{\tau_k\}_{k= 1}^\infty$ and $\{\sigma_k\}_{k=1}^\infty$ according to
\begin{equation}
\label{eq:eiv_switches}
\begin{aligned}
    \tau_1& = 0, \text{ and}  \\
    \sigma_k & = \inf\{t\geq \tau_k \!: \!\lambda_{\min}\!(B_t^{-\frac{1}{2}}U_t^\top U_tB_t^{-\frac{1}{2}}) \!<\! A(t)\}, \\
    \tau_{k+1} & = \inf\{t\geq \sigma_k\! : \!\lambda_{\min}\!(B_t^{-\frac{1}{2}}U_t^\top U_tB_t^{-\frac{1}{2}}) \!\geq \!A(t)\}, \\
\end{aligned}
\end{equation}
where $A(t) = t^\gamma \log(t)$, given in Algorithm \ref{alg:two}.

The RAID Algorithm is adapted to the endogenous-noise setting in Algorithm \ref{alg:two}. Utilizing incentive policy \eqref{eq:control} and switching schedule \eqref{eq:eiv_switches}, the planner can guarantee the a.s. convergence rate of $\tilde \Theta(t)$ to $\Theta^*$, and that the $t$-stage average regret $t^{-1}\mathcal{J}_t$ a.s. vanishes.
\begin{algorithm}
\caption{no--Regret Adaptive Incentive Design (RAID) in the Endogenous-Noise Case}\label{alg:two}
\begin{algorithmic}[1]
\Statex \textbf{Parameters:}
\Statex Parameters $\gamma \in [1/2,\,  1)$, $\sigma^2_p > 0$, and $\eta > 0$
\Statex Desired profile $(x^\dagger, u^\dagger)$ with $x^\dagger \in \operatorname{int} \mathcal{X}$
\Statex Iterations $t\geq 1$
\Statex \textbf{Output:}
\Statex Estimate and incentive $\{ \tilde \Theta(\tau), p(\tau)\}_{\tau=1}^t$
\Statex
\State Define function $A(\tau) = \tau^{\gamma}\log \tau$\;
\State Initialize $X_{1} = 0$, $U_{1}=0$, $B_{1}=0$, and $\tilde \Theta(1)$\;
\NoIndentFor{$\tau = 2, \ldots, t$}
\State \(\textrm{Explore}_{\tau} \gets
        \left(\det(B_{\tau-1}) = 0\right)
        \textbf{ or } \)
\Statex \(\qquad \qquad 
        \Big(\lambda_{\min}\!\big(
        B_{\tau-1}^{-\frac{1}{2}}
        U_{\tau-1}^\top U_{\tau-1}
        B_{\tau-1}^{-\frac{1}{2}}
        \big) < A(\tau-1)\Big)
        \)
    \If{$\textrm{Explore}_{\tau}$} \Comment{Exploration phase}
    \State Draw incentive parameter $p(\tau) \sim \mathcal N(0,\sigma^2_p\ \mathbb I)$
    \State Issue $\gamma(\cdot, p(\tau))$ thrice
    \State Observe agent responses $\hat x^{k}(\tau)$, $k\in [3]$
    \State $\delta(\tau) \gets \mathbf{1}\{d_{{\mathcal{X}}}(\hat x^{3}(\tau)) \geq \eta\}$
    \State $X_\tau \gets X_{\tau-1} + \delta(\tau) \hat p^{{1}}(\tau) \xi^{{2}}(\tau)^\top$
    \State $U_\tau \gets U_{\tau-1} + \delta(\tau) \xi^{{1}}(\tau) \xi^{{2}}(\tau)^\top$
    \State $B_\tau \gets B_{\tau-1} + \delta(\tau) \xi^{{2}}(\tau) \xi^{{2}}(\tau)^\top$
    \State Update  $\tilde \Theta(\tau)$, according to \eqref{eq:eiv_estimator}
\Else \Comment{Exploitation phase}
    \State Update $p(\tau) \gets  -\tilde\Theta(\tau-1)  \Phi(x^\dagger)$
    \State Issue incentive $\gamma(\cdot, p(\tau))$
    \State Maintain $B_{\tau} \gets B_{\tau-1}$, $U_{\tau} \gets U_{\tau-1}$, $X_\tau \gets X_{\tau-1}$
    \State Maintain $\tilde \Theta(\tau) \gets \tilde \Theta(\tau-1)$
\EndIf
\EndNoIndentFor
\end{algorithmic}
\end{algorithm}

\begin{theorem}
\label{thm:eiv_algorithm_regret}
    Let Assumptions \ref{ass:feasibility}, \ref{ass:types}, \ref{ass:eiv_noise} and~\ref{ass:exploration_margin} hold. For all  $\gamma \in [\frac{1}{2},1)$, the type estimates $\{\tilde \Theta(\tau)\}_{\tau =1}^t$ produced by Algorithm \ref{alg:two} satisfy
    \[\| \tilde \Theta(t) - \Theta^*\|_F^2 = O(t^{-\gamma}) \text{ a.s.}\]
    Moreover, the system planner's regret $\mathcal{J}_t$, as defined in \eqref{eq:eiv_regret}, satisfies
    \[ \mathcal{J}_t = O (t^\gamma \log t)
    \text{ a.s.}\]
\end{theorem}
\begin{proof}
    Presented in Appendix \ref{app:proofs_main}.
\end{proof}

It can also be shown that Algorithm~\ref{alg:two} enforces the  $\log t = o(\lambda_{\min}(B_t^{-\frac{1}{2}}U_t^\top U_t  B_t^{-\frac{1}{2}})$ growth of the information matrix and an adequate number of exploration iterations. This result is given in Lemma~\ref{lemma:eiv_number_of_exploration_samples}, Appendix~\ref{app:proofs_main}.

\section{Numerical Examples}
\label{sec:numerical_examples}
We now illustrate the performance of the proposed incentive design algorithms through numerical simulations. 
The general setup is as follows.
Consider a game played between the system planner and $n=3$ players, over the feasible set ${\mathcal{X}} = [-1,1]^3$.
Players have polynomial nominal cost functions of  $3$rd degree, as follows.
\[\begin{aligned}
    \ell_1 (x) & = 3 (\tfrac{1}{6}x_1^3 + x_1^2  + x_1x_2 - x_1x_3 + x_1 )\\
    \ell_2 (x) & = 3 (\tfrac{1}{6}x_2^3 + x_2^2  - x_1x_2 + x_2x_3 - x_2) \\
    \ell_3 (x) & = 3 (\tfrac{1}{6}x_3^3 + x_3^2  + x_1x_3 - x_2x_3 + x_3)
\end{aligned}\]
The nominal cost functions $\ell_i(x)$ lead to the parametrization of agent types according to \eqref{eq:parametrization}, where $\Phi(x): {\mathbb{R}}^3 \to {\mathbb{R}}^{10}$ is a monomial basis for $2$nd degree polynomials
\[\Phi(x) \!=\! \begin{bmatrix}
    1 & x_1  &  x_2  &  x_3 &  x_1^2  &  x_1 x_2  & x_1x_3  &  x_2^2  &  x_2x_3 &  x_3^2
\end{bmatrix}^\top,\]
and the corresponding type matrix $\Theta^* \in {\mathbb{R}}^{3\times 10}$,
\[\Theta^*\! = 3\!\begin{bmatrix}
    \!\phantom{-} 1 & \phantom{-} 2 & \phantom{-} 1 & -1 & 0.5 & \phantom{-}0 & \phantom{-}0 & \phantom{-}0   & \phantom{-}0 & \phantom{-}0 \\
    \!-1 & -1 & \phantom{-} 2 &  \phantom{-}1 & \phantom{-}0   & \phantom{-}0 & \phantom{-}0 & 0.5 & \phantom{-}0 & \phantom{-}0 \\
    \!\phantom{-} 1 &  \phantom{-}1 & -1 & \phantom{-} 2 & \phantom{-}0 & \phantom{-}0 & \phantom{-}0 & \phantom{-}0 & \phantom{-}0 & 0.5
\end{bmatrix}.\]
One can verify that type $\Theta^*$  satisfies $ \operatorname{Sym}(\Theta^* {\nabla}\Phi(x)) = \frac{1}{2}\operatorname{Diag}(2+x_1,\, 2+x_2,\,  2+x_3) \succeq \frac{1}{2}\mathbb I$ for every $x \in {\mathcal{X}}$, so Assumption~\ref{ass:types} holds with $R = \mathbb I$ and $m=1$.
The planner's social cost is selected as $\Psi(x,u) = \frac{1}{2}\|x-x^\dagger\|_2^2 + \|u\|_2^2$, where $x^\dagger = (0.5, \, 0.5, \, 0.5)^\top$.

\subsection{The Measurement-Noise Case}
We first validate the predictions of Section \ref{sec:measurement_noise} in problems in which Assumption \ref{ass:independet_noise} holds.
We take the noise process $\{e(t)\}_{t= 1}^\infty$ in \eqref{eq:model} to be i.i.d. and normally distributed with variance $\sigma_e^2 = 0.1$.
The system planner utilizes Algorithm~\ref{alg:one} with parameters $\gamma = 0.5$ and $\sigma_p^2 = 10$.

Theorem \ref{thm:algorithm_regret} predicts that the parameter estimation error $\| \hat \Theta(t)- \Theta^*\|_F^2$ decays with  $O (t^{-\gamma})$ a.s. and the normalized regret $t^{-1}\mathcal{R}_t$ decays with $O(t^{\gamma-1}\log t)$ a.s.  Figure~\ref{fig:algorithm_estimation_regret} depicts the estimation error and the planner's accumulated regret, computed across 10 independent runs of Algorithm \ref{alg:one}. The results obtained are consistent with the almost-sure predicted decay rates.
\begin{figure}
    \centering
    \includegraphics[width=0.9 \linewidth]{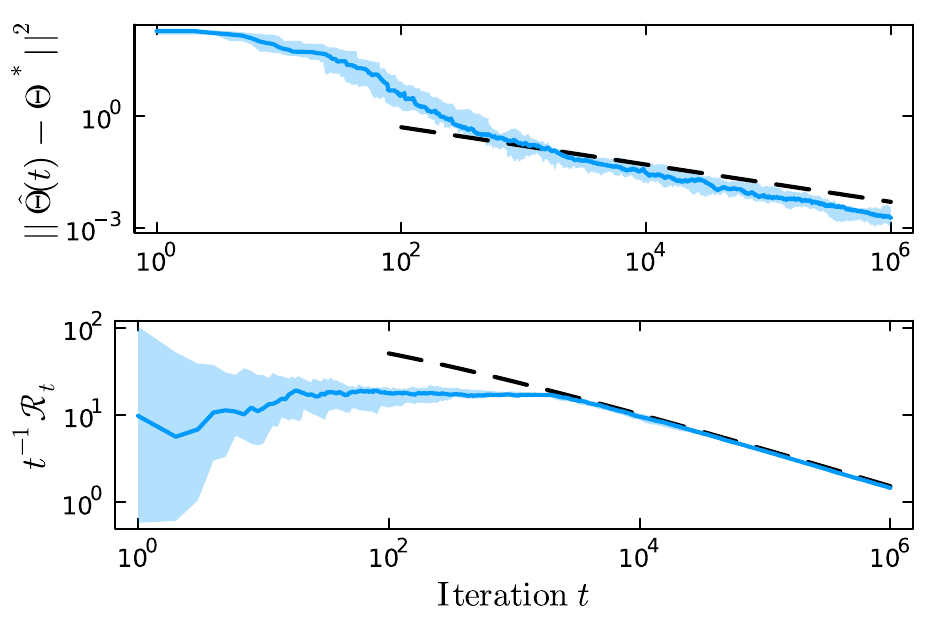}
    \caption{Median (solid) and min-max envelope (shaded) of the parameter estimation error $\|\hat \Theta(t)- \Theta^*\|_F^2$ and regret $t^{-1}\mathcal{R}_t$, produced across 10 independent runs of Algorithm \ref{alg:one}. Dashed lines indicate the a.s. convergence rates $O(t^{-\gamma})$ (top) and $O(t^{\gamma-1}\log t)$ (bottom) predicted by Theorem \ref{thm:algorithm_regret}.}
    \label{fig:algorithm_estimation_regret}
\end{figure}

Lemma~\ref{lemma:number_of_exploration_samples} predicts an a.s. $\Theta(A(t))$ growth rate for both $\lambda_{\min}(t)$ and the total number of exploration iterations $\#(t)$.
Figure \ref{fig:excitation} presents the signal $\operatorname{tr}(\Sigma_t) = \Theta(1 / \lambda_{\min}(t))$, that defines the exploration and exploitation phases,  and the aggregate number of exploration samples $\#(t)$. 
It is evident in Figure~\ref{fig:excitation} that the a.s. $\Theta(1/ A(t))$ decay of $\operatorname{tr}(\Sigma_t)$ is consistent with the established excitation rate, and so is the a.s. $\Theta(A(t))$ growth of $\#(t)$.

\begin{figure}
    \centering
    \includegraphics[width=0.9\linewidth]{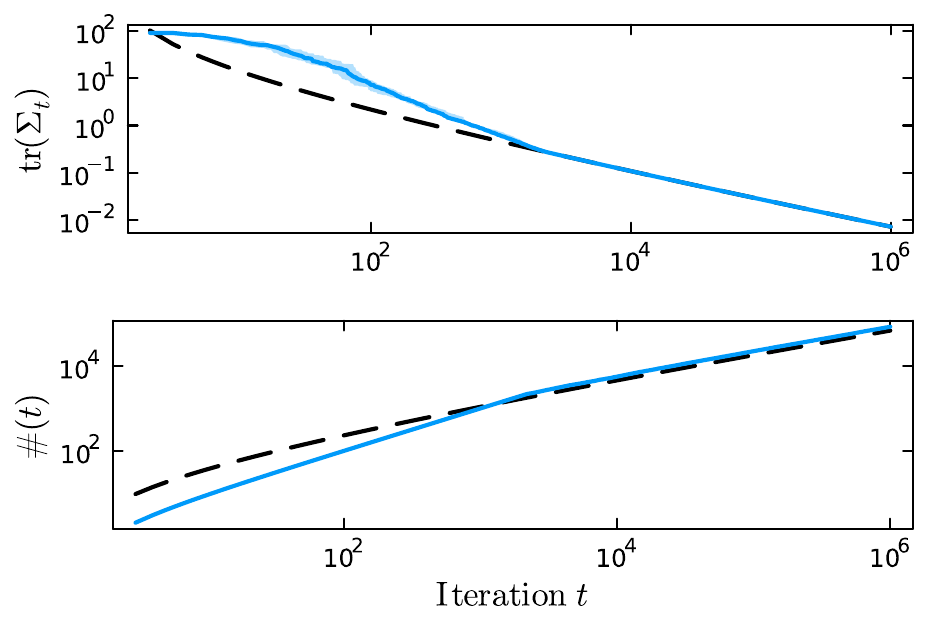}
    \caption{Median (solid) and min-max envelope (shaded)  of the excitation signal $\operatorname{tr}(\Sigma_t)$ and number of exploration samples $\#(t)$ up to iteration $t$, produced across 10 independent runs of Algorithm \ref{alg:one}. Dashed lines indicate the a.s. rates $\Theta(1/A(t))$ (top) and $\Theta(A(t))$ (bottom) predicted by Lemma~\ref{lemma:number_of_exploration_samples}.}
    \label{fig:excitation}
\end{figure}

Figure \ref{fig:different_noise}  
illustrates that the convergence of the parameter estimator is due to the injected probing excitation rather than the measurement noise $\{e(t)\}_t$.
Notably, $\hat \Theta(t)$ is strongly consistent even when the support of $e(t)$ is of measure zero (for instance, when each component $e_i(t)$ is Rademacher distributed on $\pm0.1$  as in Figure \ref{fig:different_noise}, bottom row). 
\begin{figure}
    \centering
    \includegraphics[width=0.9\linewidth]{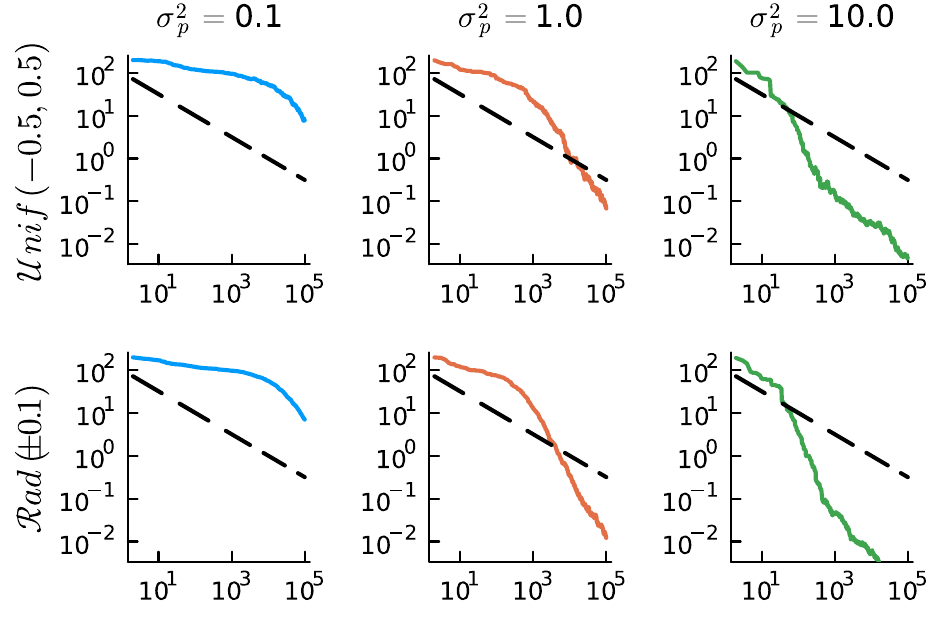}
    \caption{Trajectories of the estimation error $\|\tilde \Theta(t)- \Theta^*\|_F^2$ for different noise processes $\{e(t)\}_t$.
    Dashed lines indicate the predicted $O(t^{-0.5})$ rate.
    Columns vary the probing variance $\sigma_p^2$  in Algorithm~\ref{alg:one}.
    Measurement noise is an ${\mathbb{R}}^n$-random variable where $e_i(t)\sim \mathcal{U}nif [-0.5, 0.5]$ (top row) or Rademacher distributed on $\{\pm 0.1\}$ (bottom row).}
    \label{fig:different_noise}
\end{figure}

\subsection{The Endogenous-Noise Case}
This subsection validates the predictions of Section~\ref{sec:endogenous_noise} for problems with an endogenous noise process $\{v(t)\}_t$, where Assumptions \ref{ass:eiv_noise} and \ref{ass:exploration_margin} hold.
In~\eqref{eq:model}, we take the measurement noise $\{e(t)\}_{t= 1}^\infty$ to be i.i.d. and normally distributed with variance $0.1$.
We take the endogenous noise $\{v(t)\}_{t= 1}^\infty$ to be distributed according to the truncated normal distribution with variance $0.1$ and support on $[-0.1, 0.1]^n$.
One can verify that the set ${\mathcal{X}} = [-1,1]^n$ has diameter $|{\mathcal{X}}| = 2\sqrt{n}$, so the planner may select $\eta \in (0.4 \sqrt{n}, \sqrt{n}) $ to satisfy Assumption~\ref{ass:exploration_margin}.
The system planner utilizes Algorithm \ref{alg:two} with parameters $\gamma = 0.5$, $\eta = 0.7$ and $\sigma_p^2 = 10$.
The trajectories of estimation error $\|\tilde \Theta(t) - \Theta^*\|_F^2$ and normalized regret $t^{-1}\mathcal{J}_t$ are given in Figure \ref{fig:eivnoise_ribbons_error_regret}
and satisfy the a.s. predicted rates.

\begin{figure}
    \centering
    \includegraphics[width=0.9 \linewidth]{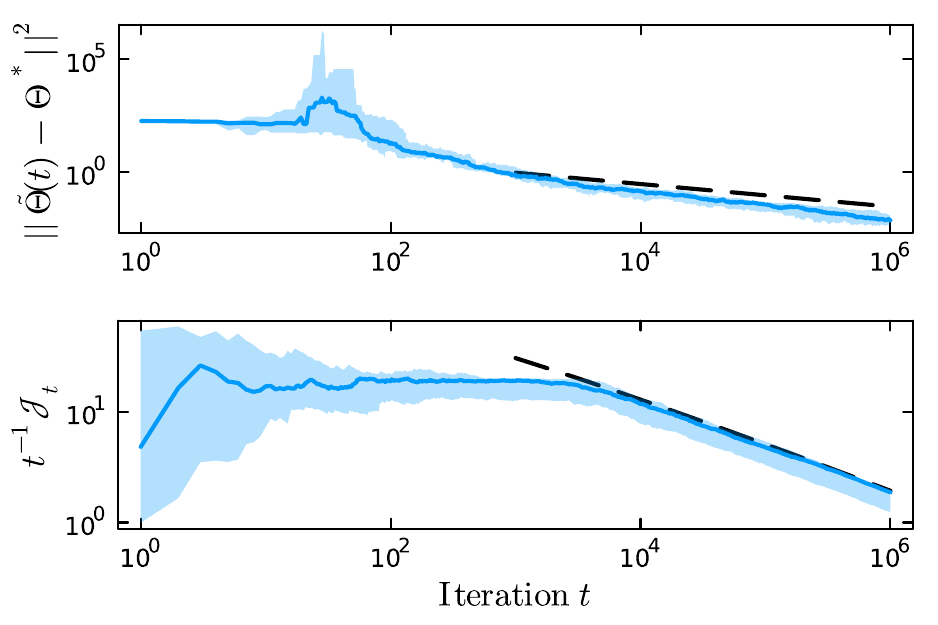}
    \caption{Median (blue) and min-max envelope (shaded) of the parameter estimation error $\|\tilde \Theta(t)- \Theta^*\|_F^2$ and regret $t^{-1}\mathcal{J}_t$, produced across 10 independent runs of Algorithm \ref{alg:two}. Dashed lines indicate the a.s. convergence rates $O(t^{-\gamma})$ (top) and $O(t^{\gamma-1}\log t)$ (bottom) predicted by Theorem \ref{thm:eiv_algorithm_regret}.}
    \label{fig:eivnoise_ribbons_error_regret}
\end{figure}

\subsection{Sensitivity to Model Misspecification}

We further investigate a two-player coupled oscillator game adapted from~\cite{ratliffAdaptiveIncentiveDesign2021}, in which the nominal costs do not lie in the span of $\Phi(x)$. Players have private costs
\begin{equation}
\label{eq:oscillator_cost}
\ell_i(x) = -a_i \cos(x_i) + \cos (x_1 - x_{2}),
\end{equation}
where $x \in {\mathcal{X}} = [-1,1]^2$ and $a_i>0$ are scalars.
The corresponding pseudo-gradient is
\[
G_0(x)
=
\begin{bmatrix}
a_1\sin(x_1)-\sin(x_1-x_2)\\
a_2\sin(x_2)+\sin(x_1-x_2)
\end{bmatrix}.
\]
The planner parameterizes the partial derivatives $\partial \ell_i / \partial x_i$ with the third-order monomial basis
\[\Phi(x) = \begin{bmatrix}
    1 & x_1 & x_2 & \ldots & x_1^3 & x_1^2x_2 &x_1x_2^2 & x_2^3
\end{bmatrix}^\top, \]
which corresponds to a Taylor approximation of the true trigonometric gradient $\partial \ell_i / \partial x_i \approx \theta_i^{*\top} \Phi(x)$. Hence, the estimator here is misspecified in the sense that it corresponds to a finite-dimensional surrogate rather than an exact representation of the cost parameters.  

We select parameters $a_1 = 6$ and $a_2 = 5$, for which $G_0$ is strongly monotone on ${\mathcal{X}}$.
We consider the desired profile $x^\dagger = (0.5, -0.2)^\top$ and $u^\dagger =(0,0)^\top$. The planner's optimal incentive is given by~\eqref{eq:optimal_policy} with
\[p^* = \begin{bmatrix}
    -\partial_{x_1} \ell_1(x^\dagger) \\ -\partial_{x_2} \ell_2(x^\dagger)
\end{bmatrix} = \begin{bmatrix}
    -a_1 \sin(x_1^\dagger) + \sin(x_1^\dagger-x_2^\dagger) \\ -a_2 \sin(x_2^\dagger) - \sin(x_1^\dagger-x_2^\dagger)
\end{bmatrix}.\]
The noise process $\{e(t)\}_{t=1}^{\infty}$ is selected as an i.i.d. normally distributed with variance $\sigma_e^2 = 0.1$, and the planner uses Algorithm~\ref{alg:one} with parameters $\gamma = 0.5$ and $\sigma_p^2 = 10$.

Figure~\ref{fig:incentive_approximation} presents the exploitation-stage incentive error $\|p(t)-p^*\|_2$. In contrast to the previous examples, the error does not converge to zero. Instead, the estimator $\hat \Theta(t)$ approaches the finite-dimensional surrogate $\Theta^*$, and the
incentive $\gamma(\cdot, p(t))$ approaches $\gamma(\cdot, p^*)$ up to a constant bias, induced by the mismatch between the true pseudo-gradient
and the polynomial model approximation. 
This indicates that the adaptive mechanism reconstructs an approximately optimal incentive,
though the guarantee $p(t) \to p^*$ no longer follows from Theorem~\ref{thm:algorithm_regret}.

\begin{figure}
    \centering
    \includegraphics[width=0.9 \linewidth]{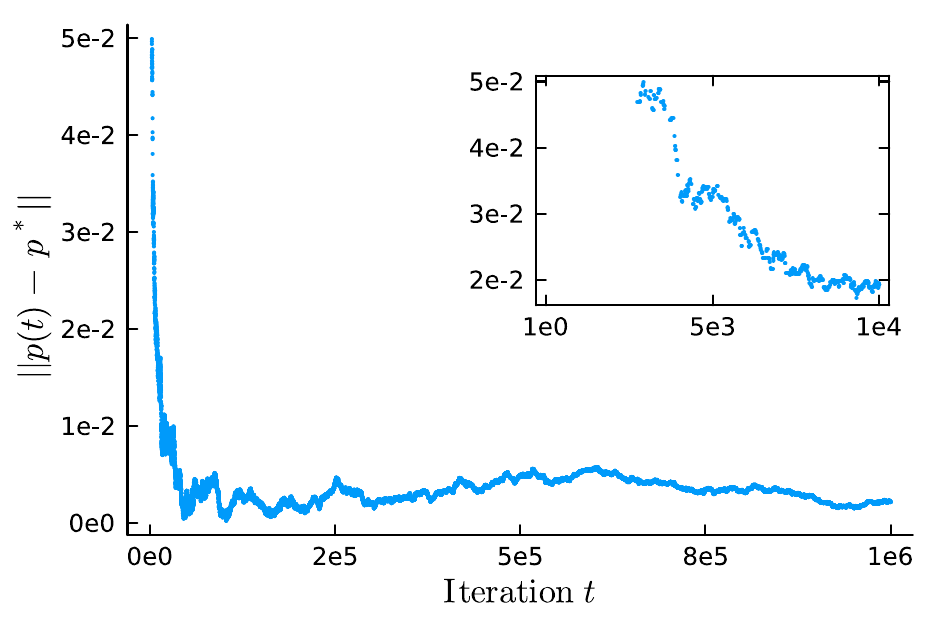}
    \caption{Incentive error $\|p(t)-p^*\|_2$ across the exploitation iterations of Algorithm~\ref{alg:one}. Because the trigonometric pseudo-gradient is approximated by a finite-dimensional surrogate, the error converges to a nonzero approximation floor.}
    \label{fig:incentive_approximation}
\end{figure}

\section{Concluding Remarks}
\label{sec:conclusion}

This work introduced the no-Regret Adaptive Incentive Design (RAID) problem and developed a unified framework for type identification and control in adaptive incentive design.
We establish that, under a diagonal strict convexity condition, an affine incentive law is sufficient and optimal within a broad class of convex incentives. Under this incentive structure, we show that the incentive–response map is unique and a diffeomorphism from an informative incentive region to the feasible response set, enabling learning schemes conditioned on informative responses.
We utilize a diminishing excitation condition to design a strongly consistent estimator 
and an incentive policy that alternates between probing (exploration) and estimate-based (exploitation) incentives. We prove that this policy 
achieves a parameter estimation rate of $O(t^{-0.5})$ and regret of $O(t^{0.5}\log t)$, almost surely. 
Finally, we extend the RAID framework to address the EIV bias induced by the presence of endogenous noise in the response model. We propose a repeated-sampling estimator and algorithm, which retain the same almost-sure parameter and regret convergence rates. 

There are three natural extensions of the work presented here. It is of interest to investigate RAID in a setting with non-differentiable pseudo-gradient mappings, where strong monotonicity might be enforced via generalized gradients or a variational-inequality formulation. This introduces new questions on the existence and uniqueness of Nash equilibria, as well as the regularity conditions that are necessary for the estimation and regret analysis. Another possible extension may consider games with multiple or no nominal equilibria. In this setting, the RAID problem would further require inducing a unique well-defined equilibrium behavior via incentive design. Future work can also consider adaptive incentive design for networks of agents, as considered in multi-agent systems~\cite{zhang18intrinsictetrahedron, zhang20anintrinsicapproach,zhang21modelingcollective}.

\appendix
\section{Proof of Main Theorems}
\label{app:proofs_main}

This appendix contains the proofs for Sec.~\ref{sec:measurement_noise} and~\ref{sec:endogenous_noise}.
 
\subsection{Proof of Lemma \ref{thm:sufficient_strong_consistency}}
\begin{proof}
     We write~\eqref{eq:model} component-wise as
     \[\hat p_i(t) = -\theta_i^{*\top}\xi(t) + e_i(t),\]
     where $\xi(t) = \Phi(x(t))$. By Assumption \ref{ass:independet_noise}, each $\{e_i(t)\}_{t=1}^\infty$ is an i.i.d., zero-mean noise process with $\sup_t \mathbb E|e_i(t)|^a < \infty$ for $a>2$.
     Estimator~\eqref{eq:solution_estimator} can be written agent-wise as
     \[\begin{aligned}
         \hat \theta_i(t) & = \operatorname*{arg\,min}_{\theta_i \in {\mathbb{R}}^D} \sum_{\tau=1}^t (\hat p_i(\tau) + \theta_i^\top \xi(\tau))^2\delta(\tau) \\
         & = - \Big(\sum_{\tau=1}^t \delta(\tau)\hat p_i(\tau)\xi(\tau)\Big)\Big(\sum_{\tau=1}^t \delta(\tau)\xi(\tau)\xi(\tau)^\top\Big)^{-1},
     \end{aligned}\]
    where $\delta(t) = \mathbf 1\{ \hat x(t) \in \operatorname{int} {\mathcal{X}}\}$.
    We denote by $\lambda_{\min}(t)$ and $\lambda_{\max}(t)$ the minimum and maximum eigenvalue, respectively, of $S_t^{-1}$, given in~\eqref{eq:Sigma_inverse_rank_one_update}.
    
    By the compactness of $\mathcal{X}$, it holds that $\lambda_{\max}(t)$ grows at most linearly with $t$, since
    \[\lambda_{\max}(t) \leq t \max_{x \in \mathcal{X}} \| \Phi(x)\|_2^2 = O(t).\]
    Consequently, when $\log t = o(\lambda_{\min}(t))$, it follows that $\log \lambda_{\max}(t) =o (\lambda_{\min}(t))$. Applying~\cite[Theorem 1]{laiLeastSquaresEstimates1982} and $\|\hat \theta_i(t) - \theta_i^*\|_2 = O(\|\hat \theta_i(t) - \theta_i^*\|_\infty)$, we conclude
    \[\|\hat \theta_i (t) -  \theta_i^*\|_2^2 = O \Big(\frac{\log(\lambda_{\max}(t))}{\lambda_{\min}(t)}\Big) = O \Big(\frac{\log t}{\lambda_{\min}(t)}\Big) \text{ a.s.}.\]
    Finally, the estimator $\hat \Theta(t)$ in \eqref{eq:solution_estimator} satisfies
    \[\|\hat \Theta(t) - \Theta^*\|_F^2 = \sum_{i=1}^n \|\hat \theta_i(t)-\theta_i^*\|_2^2 = O \Big(\frac{\log t}{\lambda_{\min}(t)}\Big) \text{ a.s.}\]
\end{proof}

\subsection{Proof of Theorem \ref{thm:information_growth}}
\begin{proof}
    Since $\mathcal{X}$ is compact, $\|\xi(t)\|_2$ is uniformly bounded and $\mathbb E \|\xi(t)\|_2 < \infty$.  
    Define $M=\mathbb E[\xi(t) \xi(t)^\top \delta(t)]$, where $M \succ 0$ due to Lemma \ref{lemma:excitation}. 

    By $\eqref{eq:Sigma_inverse_rank_one_update}$, we have 
    \(S_t^{-1} = \sum_{\tau=1}^t \delta(\tau)\xi(\tau)\xi(\tau)^\top,\)
    and applying the strong law of large numbers
    \[\begin{aligned}
        t^{-1}S_t^{-1} = t^{-1} \sum_{\tau = 1}^t \delta(\tau)\xi(\tau) \xi(\tau)^\top  \overset{a.s.}{\longrightarrow} M.
    \end{aligned} \]
    By Weyl's inequality for symmetric matrices,
    \[|t^{-1} \lambda_{\min}(S_t^{-1}) - \lambda_{\min}(M) | \leq \|t^{-1}S_t^{-1}- M\|_2,\]
    which in turn implies
    $t^{-1}\lambda_{\min}(S_t^{-1}) \overset{a.s.}{\longrightarrow} \lambda_{\min}(M)$, and hence, $\lambda_{\min}(t) = \lambda_{\min}(S_t^{-1}) = \Theta(t)$ a.s. 
\end{proof}

\subsection{Proof of Lemma \ref{lemma:number_of_exploration_samples} and Theorem \ref{thm:algorithm_regret}}
Fix a switching schedule $\{\tau_k\}_{k}$, $\{\sigma_k\}_{k}$ according to \eqref{eq:switches}.
Let $ K_{t}= \sup\{k : \tau_k \leq t\}$ and define $\#(t)$ to be the total number of exploration samples up to time $t$, 
that is,
\begin{equation}
\label{eq:hashtag}
\#(t) = \sum_{k=1}^{K_{t}-1} (\tau_{k+1} - \sigma_{k}) + \max\{0, t - \sigma_{K_{t}}\}.
\end{equation}
We preface the proof of Lemma \ref{lemma:number_of_exploration_samples} and Theorem \ref{thm:algorithm_regret} with the following lemma.

\begin{lemma}
\label{lemma:eigenvalue_threshold}
Suppose Assumptions \ref{ass:types} and \ref{ass:independet_noise} hold. 
Let 
\[t_0 = \inf\{t>0: S_t^{-1} \succ 0\}; \text{ with }\inf \emptyset = \infty.\]
Under Algorithm \ref{alg:one}, if $t_0 < \infty$, there exists a $k_0 < \infty$ such that, for all $k \geq k_0$:
\begin{enumerate}
    \item When $t \in [\tau_k, \sigma_k)$, it holds that $\lambda_{\min}(t) \geq A(t)$, 
    \item When $t \in [\sigma_k, \tau_{k+1})$, it holds that $\lambda_{\min}(t) < D A(t)$, 
\end{enumerate}
where $\lambda_{\min}(t) = \lambda_{\min}(S_t^{-1})$, $D$ is the size the feature map $\Phi(\cdot)$, and $A(t)$ is the threshold sequence in the algorithm.
\end{lemma}
\begin{proof}
    Since $S_t^{-1} \succeq S_{t_0}^{-1}\succ 0$, it holds that for all $t \geq t_0$,
    \(\lambda_{\max}(S_t) = 1/\lambda_{\min}(S_t^{-1}) = 1/\lambda_{\min}(t).\)
    Moreover, it also holds that
    \(\lambda_{\max}(S_t) \leq \sum_{i=1}^D \lambda_i(S_t) \leq D \lambda_{\max}(S_t).\)
    Hence
    \(1/\lambda_{\min}(t) \leq \operatorname{tr}(S_t) \leq D/\lambda_{\min}(t)\), for all $t \geq t_0.$

    Let $k_0 = \inf\{k: \tau_k \geq t_0\}<\infty$. 
    For each $k\geq k_0$, when $t \in [\tau_k, \sigma_k)$, we have $\operatorname{tr}(S_t) \leq A(t)^{-1}$, which implies $\lambda_{\min}(t) \geq A(t)$. When $t \in [\sigma_k, \tau_{k+1})$, we have that $\operatorname{tr}(S_t) > A(t)^{-1}$, which implies $\lambda_{\min}(t) < D A(t)$.
\end{proof}

\begin{proof}[\textbf{Proof of Lemma \ref{lemma:number_of_exploration_samples}}]

    Notice that $ \#(t) \to \infty$ a.s., otherwise, there exists some $t_0$ such that $\operatorname{tr}(S_{t_0}) \leq A(t)^{-1}$ for all $t\geq t_0$, which contradicts the monotonicity of $A(t)$.

    Let $I(t)$ denote those times $t$ that belong to exploration phases, so $|I(t)| = \#(t)$. 
    The subsequence of incentives during exploration phases is $\{p(\tau)\}_{\tau \in {I}(t)} = \{\varepsilon(\tau)\}_{\tau \in I(t)}$, where $\{\varepsilon(\tau)\}_{\tau \in I(t)}$ are i.i.d. Gaussian. 
    To see this, we may consider an auxiliary i.i.d. sequence $\{\mu(k)\}_{k= 1}^\infty$, where  $\mu(k) \sim \mathcal{N}(0,\sigma^2_p\ \mathbb I)$. Then, for each $\tau \in \mathcal I(t)$, sampling probing noise $\varepsilon(\tau)$ is equivalent to taking the noise from $\{\mu(k)\}_k$ with $\varepsilon({\tau})=\mu(\#(\tau))$. In this view, noise $\{\mu(k)\}_k$ is the uninterrupted probing sequence, and $\{\varepsilon(\tau)\}_{\tau \in I(t)}=\{\mu(\#(\tau))\}_{\tau \in \mathcal{I}(t)} = \{\mu(k)\}_{k=1}^{\#(t)}$.
    
    By Theorem \ref{thm:information_growth}, there exists a constant $c>0$ such that 
    \[\lim_{t\rightarrow \infty} \frac{1}{\#(t)}\lambda_{\min} \Big(\sum_{\tau \in \mathcal{I}(t)}\xi(\tau)\xi(\tau)^\top \delta(\tau) \Big) = c,  \text{ a.s.}\]
    Since $S_t^{-1}$ is only updated in Algorithm \ref{alg:one} at exploratory iterations $\tau \in \mathcal{I}(t)$, we have
   $ \sum_{\tau \in \mathcal{I}(t)} \xi(\tau)\xi(\tau)^\top \delta(t) = S_t^{-1}$.
    Therefore, for any $\varepsilon \in ( 0, c)$, there exists $t_\varepsilon>0$ such that
    \begin{equation}
    \label{eq:thm3_proof_1}
           0 < (c- \varepsilon) \#(t) \leq \lambda_{\min}(t)  \leq (c+\varepsilon) \#(t), \ \forall t\geq t_\varepsilon.
    \end{equation}
    Denote $c^- = c-\varepsilon$ and $c^+ = c+\varepsilon$.

    Let $B = \max_{x\in \mathcal{X}} \| \Phi(x)\|^2_2$.
    For all $t$, observe that
    \begin{equation}
    \label{eq:thm3_proof_2}
    \begin{aligned}
        \lambda_{\min}(t+1) & \leq \lambda_{\min}(t) + \lambda_{\max}\left(\xi(t)\xi(t)^\top\right) \\ & \leq  \lambda_{\min}(t) +B.
    \end{aligned}
    \end{equation}
    
    Finally, by Lemma \ref{lemma:eigenvalue_threshold}, there is a $k_\varepsilon$ such that for all $k\geq k_\varepsilon$:
    \begin{equation}
    \label{eq:thm3_proof_3}
    \begin{aligned}
        &  \lambda_{\min}(t) \geq  A(t),\  \forall t \in [ \tau_k, \sigma_k ), \\ 
        & \lambda_{\min}(t) < D A(t), \ \forall t\in [ \sigma_k, \tau_{k+1}). 
    \end{aligned}
    \end{equation}

    We can show that when $t \geq \max\{t_\varepsilon, \tau_{k_\varepsilon}\}$, \eqref{eq:thm3_proof_2}-\eqref{eq:thm3_proof_3} ensure that $\lambda_{\min}(t) = O(A( t))$. This is because
    \begin{enumerate}
        \item If $\sigma_{K_t} \leq t <  \tau_{K_t+1}$ then \(\lambda_{\min}(t) \leq  DA(t)\), by \eqref{eq:thm3_proof_3}.
        \item If $t = \tau_{K_t}$, by \eqref{eq:thm3_proof_2} and \eqref{eq:thm3_proof_3}, 
        \[\lambda_{\min}(t) \leq \lambda_{\min}(t-1) + B  \leq DA(t-1) + B. \]
        \item If $\tau_{K_t} < t < \sigma_{K_t}$, by \eqref{eq:thm3_proof_2}, 
        \[\begin{aligned}
            \lambda_{\min}(t) = \lambda_{\min}(\tau_{K_t}) & \leq DA(\tau_{K_t}-1) + B  \\&\leq DA(t-1) + B.
        \end{aligned}\]
    \end{enumerate}
    Moreover, \eqref{eq:thm3_proof_1}-\eqref{eq:thm3_proof_3} ensure $\lambda_{\min}(t) = \Omega(A( t))$ a.s. 
    \begin{enumerate}
        \item If $\tau_{K_t} \leq t < \sigma_{K_t}$, then \(\lambda_{\min}(t) \geq  A(t)\) by \eqref{eq:thm3_proof_3}, 
        \item If $\sigma_{K_t} \leq t < \tau_{K_t+1} $, by \eqref{eq:thm3_proof_1}, with probability one,
        \[\begin{aligned}
            \lambda_{\min}(t) & \geq c^- \left( \#(\sigma_{K_t}-1) + t-\sigma_{K_t}+1 \right) \\
            & \geq \frac{c^-}{c^+}\left( \lambda(\sigma_{K_t}-1) + c^+(t-\sigma_{K_t}+1) \right) \\
            & \geq \frac{c^-}{c^+}\left( A(\sigma_{K_t}-1) + c^+(t-\sigma_{K_t}+1) \right).
        \end{aligned} \]

        Observe that for any $\gamma < 1$, $t-A(t)$ is increasing in $t$, and therefore $t-\sigma_{K_t} + 1 \geq A(t) - A(\sigma_{K_t}-1)$.
        Substituting in the above, we have
        \begin{equation}
            \label{eq:thm3_proof_4}
            \begin{aligned}
                \frac{\lambda_{\min}(t)}{A(t)} & \geq \frac{c^-}{c^+} (c^+ + (1-c^+)\frac{A(\sigma_{K_t}-1)}{A(t)}) \\
                &\overset{(a)}{\geq}\min\{c^-, \frac{c^-}{c^+}\},
            \end{aligned}
        \end{equation}
        where $(a)$ holds since $0 \leq  \frac{A(\sigma_{K_t})}{A(t)} \leq 1$, and the affine function attains its minimum at either endpoint.
        Hence $\lambda_{\min}(t) = \Omega(A(t))$ a.s.
    \end{enumerate}
     
    We conclude that $\lambda_{\min}(t) = \Theta(A(t))$ a.s., and the fact that $\#(t) =\Theta(A(t))$ a.s. follows from \eqref{eq:thm3_proof_1}.
\end{proof}

\begin{proof}[\textbf{Proof of Theorem \ref{thm:algorithm_regret}}]
    Utilizing Lemma \ref{lemma:number_of_exploration_samples}, we have shown $\lambda_{\min}(t) = \Theta(A(t))$ a.s. By Lemma \ref{thm:sufficient_strong_consistency} this implies,
    \[\|\hat \Theta(t) - \Theta^*\|_F^2 = O(\frac{\log t}{\lambda_{\min}(t)})= O(\frac{\log t}{A(t)}) = O(t^{-\gamma}) \text{ a.s.}\]
    Due to Assumption~\ref{ass:feasibility}, the system planner's regret $\mathcal R_t$ defined in \eqref{eq:regret_definition} satisfies
    \begin{equation}
    \label{eq:thm_proof_temp1}
    \begin{aligned}
        & \mathcal R_t  \leq  {L_\Psi}^2 \sum_{\ell=1}^t \left( \|\hat x(\ell)- x^\dagger\|_2^2 + \|u(\ell)- u^\dagger\|_2^2 \right).
    \end{aligned}
    \end{equation}
    Substituting~\eqref{eq:optimal_policy} into $u(\ell) -u^\dagger$, one observes that
    \[\begin{aligned}
        u(\ell) -u^\dagger & = \operatorname{diag}(p(\ell)) (\hat x(\ell)-x^\dagger) \\
        & = (\operatorname{diag}(p^*)  + \operatorname{diag}(p(\ell)-p^*)) (\hat x(\ell)-x^\dagger),
    \end{aligned}\]
    where $\operatorname{diag}(p(\ell))$ (resp., $\operatorname{diag}(p^*)$) is the diagonal matrix whose elements are the elements of $p(\ell)$ (resp., $p^*$).
    Hence
    \begin{equation}
    \label{eq:thm_proof_temp2}
    \|u(\ell)-u^\dagger\|_2^2 \leq 2(\|p^*\|_\infty^2 + \|p(\ell) -p^*\|_2^2) \|\hat x(\ell)-x^\dagger\|_2^2.
    \end{equation}
    Substituting~\eqref{eq:thm_proof_temp2} into~\eqref{eq:thm_proof_temp1}, we obtain
    \begin{equation}
    \label{eq:thm_proof_temp3}
    \mathcal R_t  \!\leq \!  {L_\Psi}^2 \!\sum_{\ell=1}^t \!\left( 1 \!+ 2\|p^*\|_\infty^2\! + 2\|p(\ell)\!-\! p^*\|_2^2 \right)\|\hat x(\ell)\!-\! x^\dagger\|_2^2.
    \end{equation}
    We consider the RHS of~\eqref{eq:thm_proof_temp3} in each of the exploitation and exploration phases of Algorithm~\ref{alg:one}.

    \textit{$(i)$ Over exploitation intervals.}
    By Assumption \ref{ass:feasibility},  $p^* \in \mathcal{P}$ open, so there is an open ball $B(p^*, \epsilon) \subset \mathcal{P}$. Consequently, when $\hat \Theta(\ell) \to \Theta^*$, there is an $\ell_0$ such that $p(\ell) \in B(p^*, \epsilon)$ for all $\ell \geq \ell_0$. Hence
    \[\begin{aligned}
        \|\hat x(\ell)- x^\dagger\|^2_2 & \leq \|H(p(\ell))- H(p^*)\|^2_2 \overset{(a)}{\leq } h_2^2 \|p(\ell)-p^*\|_2^2,
    \end{aligned}\]
    where $(a)$ holds by Lemma \ref{lemma:NE_bijection_mapping}.
    Additionally, in each iteration it holds that $p(\ell) = - \hat \Theta(\ell-1) \Phi(x^\dagger)$, so 
    \[\begin{aligned}
        \|\hat x(\ell)- x^\dagger\|^2_2 & \leq  h_2^2 \|\Phi(x^\dagger)\|_2^2 \| \hat \Theta(\ell-1)- \Theta^*\|_F^2.
    \end{aligned}\]
    Over exploitation intervals,~\eqref{eq:thm_proof_temp3} implies
    \begin{align}
        \mathcal{R}_t^{\textrm{exploit}} \!& = O\Big( \sum_{\!\ell=\ell_0-1}^t\! ( \| \hat \Theta(\ell)- \Theta^*\|_F^2 + \| \hat \Theta(\ell)- \Theta^*\|_F^4)\Big) \notag \\
        & \overset{(a)}{=}\! O\Big(\sum_{\!\ell = \ell_0-1}^t \!\!\ell^{-\gamma}\Big) + O\Big(\sum_{\!\ell = \ell_0-1}^t \!\!\ell^{-2\gamma}\Big) \text{ a.s.,} \label{eq:thm_proof_temp4}
    \end{align}
    where $(a)$ follows from $ \| \hat \Theta(\ell)- \Theta^*\|_F^2 = O(\ell^{-\gamma})$ a.s. For $\gamma\in[0.5,1 )$,~\eqref{eq:thm_proof_temp4} implies that $\mathcal{R}_t^{\textrm{exploit}} = O(t^{1-\gamma})$ a.s.

    \textit{$(ii)$ Over exploration intervals.}
    ${\mathcal{X}}$ is compact so $\|\hat x(\ell)-x^\dagger\|_2^2 \leq \max_{x\in {\mathcal{X}}}\|x-x^\dagger\|_2^2$. Then~\eqref{eq:thm_proof_temp3} implies
    \[\mathcal{R}_t^\textrm{explore} = O(\#(t)) + O\Big(\sum_{\ell=1}^{\#(t)} \|p(\ell)-p^*\|_2^2\Big),\]
    where $\#(t)$ is the total number of exploration steps up to time $t$, given in~\eqref{eq:hashtag}. Lemma~\ref{lemma:number_of_exploration_samples} establishes that $\#(t) = \Theta(t^\gamma \log t)$ a.s. for $\gamma \in [0.5,1)$. 
    We only need then consider the term $\sum_{\ell=1}^{\#(t)} \|p(\ell)-p^*\|_2^2$ when $p(\ell) \sim \mathcal{N}(0,\sigma_p^2 \mathbb I)$. Observe that
    \(\mathbb{E}\|p(\ell)-p^*\|_2^2 = \|p^*\|_2^2 + \sigma_p^2 n < \infty,\)
    and the random variable has a finite second moment. By the strong law of large numbers,
    \[\frac{1}{\#(t)}\sum_{\ell=1}^{\#(t)}\|p(\ell)-p^*\|_2^2 \to \|p^*\|_2^2 + \sigma_p^2 n, \]
    so the term is also $O(\#(t))$ a.s. Conclude that $\mathcal{R}_t^\textrm{explore} = O(t^\gamma \log t)$ a.s.

    \textit{$(iii)$ Over all iterations.}
    By the previous observations,
    \[\begin{aligned}
        \mathcal{R}_t & \leq \mathcal{R}_t^\textrm{explore} + \mathcal{R}_t^\textrm{exploit} = {O}(t^\gamma \log t + t^{1-\gamma}) \text{ a.s.}
    \end{aligned}\]
    For $\gamma \in [\frac{1}{2}, 1)$, this gives $\mathcal{R}_t = O(t^\gamma \log t)$ a.s.
\end{proof}

\subsection{Proof of Theorem \ref{thm:eiv_sufficient_strong_consistency}}
\begin{proof}
    Whenever $\delta (t) =1$, it follows that both observations $\hat x^{1}(t), \hat x^{2}(t) \in \operatorname{int}{\mathcal{X}}$. Therefore, by~\eqref{eq:model_eiv_2},
    \[\hat p^{1}(t) = - \Theta^* \xi^{1}(t) + e^{1}(t) - v^{1}(t).\]
    Multiplying by $\delta(t)\xi^{2}(t)^\top$ and summing over all exploration periods, we obtain
    \[X_t = -\Theta^* U_t -V_t,\]
    where $V_t = \sum_{\tau=1}^t \delta(\tau) s^{1}(\tau) \xi^{2}(\tau)^{\top}$. 
    By Lemma~\ref{lemma:eiv_excitation}, there exists some $t_0$ such that $U_t$ is nonsingular for all $t \geq t_0$, so
    it follows from~\eqref{eq:eiv_estimator} that
    \(\tilde \Theta(t) - \Theta^* = V_t U_t^+, \)
    and
    \[\|\tilde \Theta(t) - \Theta^*\|_2^2 = \lambda_{\max} \left( V_t (U_t^\top U_t)^+ V_t^\top\right).\]
   On the event $\{\log t = o (\lambda_{\min}(B_t))\}$ there is some $t_1$ such that $B_t \succeq B_{t_1} \succ 0$, for all $t \geq t_1$, and hence
    \[\begin{aligned}
        \|\tilde \Theta(t) - \Theta^*\|_2^2 & = \lambda_{\max} \left( V_t B_t^{-\frac{1}{2}}B_t^{\frac{1}{2}} (U_t^\top U_t)^+B_t^{\frac{1}{2}} B_t^{-\frac{1}{2}}V_t^\top\right) \\
        & \leq \lambda_{\max}(B_t^{\frac{1}{2}} (U_t^\top U_t)^+B_t^{\frac{1}{2}}) \ \lambda_{\max}(V_t B_t^{-1}V_t^\top) \\
        & \overset{(a)}{=} \frac{\lambda_{\max}(V_t B_t^{-1}V_t^\top)}{\lambda_{\min}^+(B_t^{-\frac{1}{2}} U_t^\top U_t B_t^{-\frac{1}{2}})},
    \end{aligned}\]
    where $(a)$ holds by the fact $\lambda_{\max}(A) = (\lambda_{\min}(A^+))^{-1}$ for $A =A^\top \succeq 0$.
    It remains to show that $\lambda_{\max}(V_t B_t^{-1}V_t^\top) = O(\log t)$ a.s.
    Observe
    \[\lambda_{\max}(V_t B_t^{-1}V_t^\top) \leq \operatorname{tr}(V_t B_t^{-1}V_t^\top) = \sum_{i=1}^n v_{i,t} B_t^{-1}v_{i,t}^\top, \]
    where $v_{i,t} = \sum_{\tau=1}^t \delta(\tau) s_i^{1}(\tau) \xi^{2}(\tau)^{\top}$, represents the $i$-th row of $V_t$ and $s^{1}(\tau) = (s_1^{1}(\tau, \ldots, s_n^{1}(\tau))^\top$.
    Each of the processes $\{s_i^{1}(\tau)\}_{\tau=1}^\infty$, is i.i.d., zero-mean and satisfies $\sup_t \mathbb E (|s_i^{1}(t)|^a) < \infty$ for some $a > 2$. 
    Moreover, if we define the filtration $\mathcal{F}_t = \sigma\{v^{k}(\tau), e^{k}(\tau): \tau \leq t, k\in [3]\}$, then 
    \(\mathbb E[\delta(t) s_i^{1}(t) \xi^{2}(t) \mid \mathcal{F}_{t-1}] =0,\)
    since $\delta(t)\xi^{2}(t)$ depends on  $\{p(t), v^{2}(t), \hat v^{3}(t)\}$, and $s^{1}(t)$ is zero-mean and independent to $v^{2}(t)$, and $v^{3}(t)$.
    By~\cite[Lemma 1.(iii)]{laiLeastSquaresEstimates1982}, on the event that $\lambda_{\min}(B_t) \to \infty$, we have $v_{i,t} B_t^{-1}v_{i,t}^\top = O (\log \lambda_{\max}(B_t))$ a.s. 
     By the boundedness of $\mathcal{X}$,
    \[\lambda_{\max}(B_t) \leq t \max_{x \in \mathcal{X}} \| \Phi(x)\|_2^2 = O(t),\]
    and the result follows.
\end{proof}

\subsection{Proof of Theorem \ref{thm:eiv_algorithm_regret}}
Consider some switching schedule $\{\tau_k\}_{k= 1}^\infty$, $\{\sigma_k\}_{k=1}^\infty$ according to \eqref{eq:eiv_switches}, and 
let $ K_{t}= \sup\{k : \tau_k \leq t\}$. Further, define $\#(t)$ to be the total number of exploration samples up to iteration $t$ 
as in~\eqref{eq:hashtag}.
We preface the proof of Theorem \ref{thm:eiv_algorithm_regret} with the following lemmas.

\begin{lemma}
\label{lemma:eiv_properties}
    Let Assumptions \ref{ass:types} and \ref{ass:eiv_noise} hold, and define the matrices $B_t$ and $U_t$ as in Theorem \ref{thm:eiv_sufficient_strong_consistency}. Let
    \[t_0 = \inf\{t>0: B_t \succ 0\}; \inf \emptyset = \infty,\] 
    and $L = \max_{x\in {\mathcal{X}}} \|\Phi(x)\|_2$.
    If $t_0 < \infty$, then for all $t \geq t_0$,
    \[\lambda_{\min}(B_{t+1}^{-\frac{1}{2}}U_{t+1}^\top U_{t+1} B_{t+1}^{-\frac{1}{2}}) \leq \lambda_{\min}(B_{t}^{-\frac{1}{2}}U_{t}^\top U_{t} B_{t}^{-\frac{1}{2}}) + L^2.\]
\end{lemma}
\begin{proof}
    When $\delta(t)=0$, 
    the inequality holds trivially.
    When $\delta(t)=1$, we have
    \begin{align}
        \lambda_{\min}(B_{t}^{-\frac{1}{2}}U_{t}^\top U_{t} B_{t}^{-\frac{1}{2}}) 
        & = \min_{y \neq 0}\frac{y^\top U_{t}^\top U_{t} y}{y^\top B_t y}, \notag \\
        &  \leq \frac{y_{t-1}^\top U_{t}^\top U_{t} y_{t-1}}{y_{t-1}^\top B_t y_{t-1}},
        \label{eq:eiv_properties_temp1}
    \end{align}
    where $y_{t-1}$ is the eigenvector associated with $\lambda_{\min}(\cdot)$ at iteration $t-1$, that is,
    \begin{equation*}
        y_{t-1}\! = \!\operatorname*{arg\,min}_{y\neq 0}\frac{y^\top U_{t-1}^\top U_{t-1} y}{y^\top B_{t-1} y}\! =\! \operatorname*{arg\,min}_{y^\top B_{t-1} y = 1}{y^\top U_{t-1}^\top U_{t-1} y}.
    \end{equation*}
    Observe that when $\delta(t)=1$,
    \begin{align}
        &B_{t}  = B_{t-1} + \xi^{2}(t)\xi^{2}(t)^{\top}, \label{eq:eiv_properties_temp2} \\
        & U_{t}^\top U_{t} = U_{t-1}^\top U_{t-1}   \!+\! \xi^{2}(t)\xi^{2}(t)^{\top}U_{t-1}
        \notag \\
        & \quad +  U_{t-1}^\top\xi^{1}(t)\xi^{2}(t)^{\top} + \|\xi^{1}(t)\|_2^2 \xi^{1}(t)\xi^{2}(t)^{\top}.     \label{eq:eiv_properties_temp3}
    \end{align}
    Substituting \eqref{eq:eiv_properties_temp2} into the denominator of \eqref{eq:eiv_properties_temp1},
    \begin{equation}
    \label{eq:eiv_properties_temp4}
        \begin{aligned}
        y_{t-1}^\top B_ty_{t-1} & = y_{t-1}^\top B_{t-1}y_{t-1} + (y_{t-1}^\top\xi^{2}(t))^2 \\
        &= 1 +(y_{t-1}^\top\xi^{2}(t))^2.
    \end{aligned} 
    \end{equation}
    Substituting  \eqref{eq:eiv_properties_temp3} in the numerator of \eqref{eq:eiv_properties_temp1},
    \begin{equation}
    \label{eq:eiv_properties_temp5}
    y_{t-1}^\top U_{t-1}^\top U_{t-1} y_{t-1} \!=\! \lambda_{\min}\!(B_{t-1}^{-\frac{1}{2}}U_{t-1}^\top U_{t-1} B_{t-1}^{-\frac{1}{2}}),
    \end{equation}
    and
    \begin{equation}
    \label{eq:eiv_properties_temp6}
    \begin{aligned}
        &y_{t-1}^\top \xi^{2}(t)\xi^{1}(t)^{\top} U_{t-1}y_{t-1} \\&
        \leq |y_{t-1}^\top \xi^{2}(t)|\,  \|\xi^{1}(t)\|_2 \| U_{t-1}y_{t-1}\|_2 .
    \end{aligned}
    \end{equation}
    Let $L = \max_{x\in {\mathcal{X}}}\|\Phi(x)\|_2$. Denote $a = y_{t-1}^\top \xi^{2}(t)$  and $\lambda(t) = \lambda_{\min}(B_{t}^{-\frac{1}{2}}U_{t}^\top U_{t} B_{t}^{-\frac{1}{2}})$. By~\eqref{eq:eiv_properties_temp4},\eqref{eq:eiv_properties_temp5}, and~\eqref{eq:eiv_properties_temp6},
    \[\lambda(t) \leq \frac{\lambda(t-1) + 2|a| L \sqrt{\lambda(t)} + L^2 a^2}{1+ a^2}.\]
    Finally,
    \[\begin{aligned}
        \lambda(t)&- \lambda(t-1)  \leq \frac{-a^2\lambda(t-1) + 2|a| L \sqrt{\lambda(t)} + L^2 a^2}{1+ a^2}\\
        & \overset{(a)}{\leq} \frac{-a^2\lambda(t-1) + a^2\lambda(t-1) + L^2 + L^2 a^2}{1+ a^2}  = L^2,
    \end{aligned} \]
    where $(a)$ holds by  $(\sqrt{a^2\lambda(t-1)} - L)^2 = a^2 \lambda(t-1) + L^2 -2 L\sqrt{a^2 \lambda(t-1)} \geq 0$.
\end{proof}

\begin{lemma}
\label{lemma:eiv_number_of_exploration_samples}
Suppose Assumptions \ref{ass:types}, \ref{ass:eiv_noise}, and~\ref{ass:exploration_margin} hold.
Under Algorithm \ref{alg:two}, it holds $\lambda_{\min}(B_t^{-\frac{1}{2}}U_t^\top U_t B_t^{-\frac{1}{2}}) = \Theta(A(t))$ a.s. and $\#(t) = \Theta(A(t))$ a.s.
\end{lemma}
\begin{proof}
    The result follows from Lemmas \ref{lemma:eiv_excitation} and \ref{lemma:eiv_properties}, and the switching times \eqref{eq:eiv_switches}, by the procedure in Lemma \ref{lemma:number_of_exploration_samples}.
\end{proof}

\begin{proof}[\textbf{Proof of Theorem \ref{thm:eiv_algorithm_regret}}]

    As established in Lemma \ref{lemma:eiv_number_of_exploration_samples}, it holds that $\lambda_{\min}(B_t^{-\frac{1}{2}}U_t^\top U_t B_t^{-\frac{1}{2}}) = \Theta(A(t))$ a.s. Utilizing Theorem \ref{thm:eiv_sufficient_strong_consistency}, it follows that
    \[\|\tilde \Theta(t) - \Theta^*\|_F^2 = O\Big(\frac{\log t}{\lambda_{\min}(\cdot)}\Big)= O\Big(\frac{\log t}{A(t)}\Big) = O(t^{-\gamma}) \text{ a.s.}\]
    By the Lipschitzness in Assumption~\ref{ass:feasibility}, the system planner's regret in~\eqref{eq:eiv_regret} satisfies
    \[\mathcal{J}_t \leq {L_\Psi}^2 \sum_{\ell=1}^t(\|\hat x^{1}(\ell)-x_r^{1}(\ell)\|_2^2 + \| u^{1}(\ell)-u_r^{1}(\ell)\|_2^2),\]
    where $x_r^{k}(t) = H(p^*+v^{k}(t))$ and $u_r^{k}(t) = \gamma(x_r^{1}(t), p^*)$.
    Similarly to the proof of Theorem~\ref{thm:algorithm_regret}
    we obtain
    \begin{equation}
    \label{eq:eiv_algorithm_regret_temp1}
    \mathcal J_t  \!\leq \!  {L_\Psi}^2 \!\sum_{\ell=1}^t \!\left( 1 \!+ 2\|p^*\|_\infty^2\! + 2\|p(\ell)\!-\! p^*\|_2^2 \right)\|\hat x^{1}(\ell)\!-\! x^{1}_r(\ell)\|_2^2,
    \end{equation}
    
    We consider the RHS of~\eqref{eq:eiv_algorithm_regret_temp1} in each of the exploitation phases
    of Algorithm~\ref{alg:two}.
    By Lemma~\ref{lemma:NE_bijection_mapping}, it holds that
    \[\begin{aligned}
        \|\hat x^{1}-x_r^{1}\|_2^2 & = \|H(p(\ell)+v^{1}(\ell)) - H(p^*+v^{1}(\ell))\|_2^2\\
        &\leq h_2^2\|p(\ell) - p^* \|_2^2.
    \end{aligned}\]
    Additionally, in each exploitation iteration the planner sets $p(\ell) = - \tilde \Theta(\ell-1) \Phi(x^\dagger)$, so by arguments made previously,
    we conclude that 
    \begin{align}
        \mathcal{J}_t^{\textrm{exploit}}\! \!& = \!O\Big(\! \sum_{\ell=\ell_0-1}^t ( \| \tilde \Theta(\ell)- \Theta^*\|_F^2 + \| \tilde \Theta(\ell)- \Theta^*\|_F^4)\Big) \notag \\
        & \overset{(a)}{=}\! O\Big(\sum_{\!\ell = \ell_0-1}^t \!\!\ell^{-\gamma}\Big) \!+ O\Big(\sum_{\!\ell = \ell_0-1}^t \!\!\ell^{-2\gamma}\Big) \text{ a.s.,} \label{eq:eiv_algorithm_regret_temp2}
    \end{align}
    where $(a)$ follows from $ \| \tilde \Theta(\ell)- \Theta^*\|_F^2 = O(\ell^{-\gamma})$ a.s. For $\gamma\in[0.5,1 )$,~\eqref{eq:eiv_algorithm_regret_temp2} implies that $\mathcal{J}_t^{\textrm{exploit}} = O(t^{1-\gamma})$ a.s.

    The exploration-phase bound is identical to that in Theorem~\ref{thm:algorithm_regret} and is omitted for brevity.
    Hence,
    \[\mathcal{J}_t = \mathcal{O}(t^\gamma \log t + t^{1-\gamma}) \text{ a.s.},\]
    and 
    for $\gamma \in [\frac{1}{2}, 1)$, this implies $\mathcal{J}_t = O(t^\gamma \log t)$ a.s.
\end{proof}

\section{Proof of Auxiliary Results}
\label{app:proofs_auxiliary}
This appendix provides auxiliary lemmas on the game pseudo-gradient, and the excitation of regressors under the probing incentives. These lemmas support the main results presented in Appendix \ref{app:proofs_main}.

\subsection{Proof of Lemma~\ref{lem:linear_inccentive_is_optimal} and Lemma~\ref{lemma:NE_bijection_mapping}}
\begin{proof}[\textbf{Proof of Lemma~\ref{lem:linear_inccentive_is_optimal}}]
    We first establish that $\gamma^*(\cdot)$ is a solution to Problem~\ref{problem:ID}.
    Observe that
    (i) under this incentive, there exists a diagonal matrix $R \in \mathbb R_+^{n \times n}$ with positive diagonal entries such that $\operatorname{Sym}(R {\nabla} G_\gamma(x)) \succeq \operatorname{Sym}( R {\nabla} G_0(x)) \succ 0, \forall x\in \mathcal X$. By~\cite[Theorem 2]{rosenExistenceUniquenessEquilibrium1965a}, this is sufficient for the existence and uniqueness of the NE in the incentivized game; 
    (ii) $G_\gamma(x^\dagger) = G_0(x^\dagger) + p^*=0$, where $p^*=(p_i^*)_{i\in[n]}$; (iii) the implemented payment $\gamma^*(x^\dagger)=u^\dagger$.

    To establish optimality, let $\gamma\in\Gamma$ be a solution to Problem~\ref{problem:ID}.  Since $\gamma_i(\cdot)$ is convex in $x_i$, it holds that $\gamma_i(x)= \gamma_i(x_i)\geq \gamma_i(x_i^\dagger) + \frac{\partial}{\partial x_i} \gamma_i(x^\dagger)(x_i- x_i^\dagger), \forall x\in \mathcal X$. Assertion \eqref{eq:optimal_policy_general_0} then follows from the facts that $\gamma_i(x^\dagger_i)=u_i^\dagger$ and $\frac{\partial}{\partial x_i} \gamma_i(x^\dagger)=-\frac{\partial}{\partial x_i} \ell_i(x^\dagger)=p_i^*$. 
\end{proof}

\begin{proof}[\textbf{Proof of Lemma~\ref{lemma:NE_bijection_mapping}}]
We prove the lemma in three steps.
$(i)$ First, we show the uniqueness of the response map $x^*(\cdot)$. 
For any $x,y\in {\mathcal{X}}$, denote $v = x-y$. Since ${\mathcal{X}}$ is convex, $x_t = y + tv \in {\mathcal{X}}$ for all $t\in [0,1]$. Under Assumption \ref{ass:types},
\[v^\top R (G_0(x) - G_0(y)) = \int_0^1 v^\top R{\nabla} G_0(x_t) v dt \geq \frac{m}{2}\|v\|_2^2,\]
and then
\(
(G_0(x) - G_0(y))^\top R (x-y) \geq \frac{m}{2}\|x-y\|_2^2,    
\)
i.e., $RG_0 (x)$ is strongly monotone on ${\mathcal{X}}$. Further, for any $p\in \mathbb R^n$, $RG_\gamma(x)$ is strongly monotone, which implies the incentivized game has a unique NE on $\mathcal X$ \cite[Theorem~2]{rosenExistenceUniquenessEquilibrium1965a}. This implies the uniqueness of the response map.

$(ii)$ Next, we show that the map $x^*(\cdot)$ restricted to $\mathcal P$ is a diffeomorphism. By Assumption \ref{ass:types}, ${\nabla} G_0(x)$ is nonsingular for every $x\in \operatorname{int}{\mathcal{X}}$, so $G_0$ is a local diffeomorphism on $\operatorname{int}{\mathcal{X}}$. 
Strong monotonicity implies the injectivity of $G_0$ on $\mathcal X$. Moreover, on $\mathcal{P} = - G_0(\operatorname{int} {\mathcal{X}})$, $G_0:\operatorname{int} {\mathcal{X}} \to -\mathcal{P}$ is surjective and thus a global diffeomorphism.
Consequently, $G_0$ admits a $\mathcal{C}^1$ inverse $G_0^{-1}: -\mathcal{P} \to \operatorname{int}{\mathcal{X}}$, 
\[
H(p) := G_0^{-1}(-p),\quad  \forall p \in \mathcal{P}.
\]
Then $H:\mathcal P\to\operatorname{int}\mathcal X$ is a diffeomorphism and satisfies $G_0(H(p)) = -p$, which is precisely \eqref{eq:first_order_condition}.
Moreover, for any $p\in \mathbb R^n$ such that the response $x^*(p)$ is in $\operatorname{int} {\mathcal{X}}$, it holds that $G_\gamma(x^*(p))=G_0(x^*(p)) + p = 0$. Hence, for all $p\in \mathcal P$, $x^*(p)=G_0^{-1}(-p)=H(p)$.

$(iii)$ To show \eqref{eq:H_singular_values_bounded}, differentiating $H(p)$ gives ${\nabla} H(p) = -{\nabla} G_0(x)^{-1}$, where $x = H(p)$ is in $\operatorname{int} {\mathcal{X}}$. 
Let $R \succ 0$ be the matrix in Assumption \ref{ass:types} and $u \in {\mathbb{R}}^n\setminus\{0\}$.
We have
\[\begin{aligned}
    m \|u\|_2^2 & \leq u^\top(R {\nabla} G_0(x) + {\nabla} G_0(x)^{\top} R)u \\
    & = 2 u^\top R {\nabla} G_0(x) u \le
2\|Ru\|_2\,\|{\nabla} G_0(x)u\|_2.
\end{aligned} \]
 Therefore, $\|{\nabla} G_0(x)u\|_2 \ge \frac{m}{2\lambda_{\max}(R)}\|u\|_2$,
which implies that $\sigma_{\min}({\nabla} G_0(x)) \ge \frac{m}{2\lambda_{\max}(R)}$.
Moreover, 
\begin{align*}
\|{\nabla} H(p)\|_2 &= \|{\nabla} G_0(p)^{-1}\|_2=\sigma_{\max}({\nabla} G_0(p)^{-1}) \\
&= \frac{1}{\sigma_{\min}({\nabla} G_0(x))} \leq \frac{2\lambda_{\max }(R)}{m }=h_2 .
\end{align*}   
Likewise, by the compactness of ${\mathcal{X}}$, it holds that
\begin{align*}
&\sigma_{\min}({\nabla} H(p)) = \sigma_{\min}(-{\nabla} G_0(x)^{-1})\\
&=(\sigma_{\max}({\nabla} G_0(x)))^{-1} \geq (\max_{x\in X} \sigma_{\max}({\nabla} G_0(x)))^{-1}=h_1.
\end{align*}
\end{proof}

\subsection{Proof of Lemma \ref{lemma:excitation}}
We preface the proof of Lemma \ref{lemma:excitation} with two results.

\begin{lemma}
\label{lemma:M_matrix}
Let $\Phi: \mathbb R^n \to \mathbb R^D$ be the monomial map defined in~\eqref{eq:Phi_map}. 
For constants $0<\mu_1\le \mu_2$, let 
\[\mathcal{A}\! = \!\{A \in {\mathbb{R}}^{n \times n}\!: \!\mu_1 \|w\|_2\! \leq \|Aw\|_2 \!\leq \mu_2 \|w\|_2, \forall w \in {\mathbb{R}}^n\},\]
and $\mathcal{B} \subset {\mathbb{R}}^n$ compact.
For every $A \in \mathcal{A}$ and $b\in \mathcal{B}$,
there is a $M(A, b) \in \mathbb R^{D\times D}$ such that $\Phi(Ax+b) = M(A, b)^\top \Phi(x)$. Moreover, there exists $\bar c > 0$ such that $\|M(A,b)w \|_2 \geq \bar c \|w\|_2 $, for all $w \in {\mathbb{R}}^D$.
\end{lemma}
\begin{proof}
Let $P_{n,d}$ denote the vector space of $n$-variate polynomials of degree at most $d$. $\Phi(\cdot)$ is a basis, so every polynomial $p\in P_{n,d} $ has a unique representation $p(x) = c^\top \Phi(x)$, $c\in {\mathbb{R}}^D$.

Let $T:P_{n,d} \to P_{n,d}$ be $(Tp)(x) = p(Ax+b)$. Since $A\in\mathcal A$, $A$ is invertible. $T$ is linear and invertible, with the inverse $(T^{-1}p)(x) = p(A^{-1}(x-b))$, so $T$ is an automorphism.
Therefore, there exists a unique matrix $M$ such that the application of $T$ on $p$ is represented as
\[(Tp)(x) = (Mc)^\top \Phi(x) = c^\top \Phi(Ax+b),\]
and hence $M^\top \Phi(x) = \Phi(Ax+b)$.

Matrix $M$ is invertible, otherwise, there exists a $p\in P_{n,d}$, such that $(Tp)(x) = (Mc)^\top \Phi(x) = 0$, which contradicts that  $\ker T = \{0\}$. Finally, $\mathcal{A}\times \mathcal{B}$ is compact and $(A, b) \mapsto M(A,b)$ is continuous since each entry $[M(A,b)]_{ij}$ is polynomial in $A,b$. Therefore, $M(A,b)$  attains a minimum singular value over $\mathcal{A}\times\mathcal{B}$.
\end{proof}

\begin{lemma}
\label{lemma:PD_monomial_moments}
Let \(\Phi:\mathbb R^n\to\mathbb R^D\) be the monomial map defined
in~\eqref{eq:Phi_map}. Let $x$ be a random variable in ${\mathbb{R}}^n$ and \(A\) be an event with \(\mathbb P(A)>0\).
If the support of \(x\mid A\) contains a nonempty open set, then
\[
    \mathbb E[\Phi(x)\Phi(x)^\top\mid A]\succ0.
\]
\end{lemma}

\begin{proof}
Let \(z\in\mathbb R^D\setminus\{0\}\). Then
\(
    z^\top\mathbb E[\Phi(x)\Phi(x)^\top\mid A]z
    =
    \mathbb E[(z^\top\Phi(x))^2\mid A],
\)
where \(T_z(x)=z^\top\Phi(x)\) is a non-trivial polynomial with respect to $x$. 

By assumption, \(\operatorname{supp}(x\mid A) \supset B\), where $B$ is nonempty and open. Since \(T_z\) is continuous and nonzero, there is a point
\(x_0\in B\) such that \(T_z(x_0)\neq0\), and a neighborhood \(U_{x_0}\) of \(x_0\) such
that
\(|T_z(x)|\ge c > 0\) for all \(x\in U_{x_0}.\)
Because \(U_{x_0} \cap \operatorname{supp}(x\mid A) \neq \emptyset\), we have 
\(\mathbb P(x\in U_{x_0}\mid A)>0.\)
Therefore,
\[
    \mathbb E[(T_z(x))^2\mid A]
    \ge
    c^2\mathbb P(x\in U_{x_0}\mid A)>0.
\]
Since this holds for every \(z\neq0\), the result follows.
\end{proof}

\begin{proof}[\textbf{Proof of Lemma~\ref{lemma:excitation}}]
Notice that, due to Lemma \ref{lemma:NE_bijection_mapping}, 
\[
\begin{aligned}
& \mathbb E(\xi(t)\xi(t)^\top \delta(t))\\  
= \, &\mathbb P (\hat x(t) \in \operatorname{int} \mathcal{X} )\mathbb E(\xi(t)\xi(t)^\top \mid \hat x(t) \in \operatorname{int} \mathcal{X} ) \\ 
= \,& 
\mathbb P (p(t) \in \mathcal{P} )\mathbb E(\xi(t)\xi(t)^\top \mid p(t) \in \mathcal{P}),
\end{aligned}\]
and it is sufficient that $\mathbb E(\xi(t)\xi(t)^\top \mid p(t) \in \mathcal{P}) \succ 0$.
From Assumption~\ref{ass:feasibility}, $p^* = -\Theta^* \Phi(x^\dagger)$ is interior point to the open set $\mathcal{P}$, so, there is some open ball $B(p^*, \varepsilon) \subset \mathcal{P}$.
Then
\[\begin{aligned}
    & \mathbb E(\xi(t)\xi(t)^\top \mid p(t) \in \mathcal{P}) \\ &\succeq \mathbb P\big( p(t) \!\in B\mid p\!\in \mathcal{P}\big)E\big(\xi(t)\xi(t)^\top \mid p(t) \!\in B\big),
\end{aligned}\]
where $B = B(p^*, \varepsilon)$.
Denote $ Q =E\big(\xi(t)\xi(t)^\top \mid p(t) \in B(p^*, \varepsilon)\big)$, so it is sufficient that $Q \succ 0$. We omit the dependence on $t$ for ease of notation.

When $p\in B(p^*, \varepsilon)$, Lemma \ref{lemma:NE_bijection_mapping} guarantees that $\hat x = H(p)$ where $H \in \mathcal{C}^1$.
By the mean value theorem,
\begin{equation}
\label{eq:excitation_temp0}
H(p) = H(p^*) + \Big(\int_{0}^1{\nabla} H(\gamma_\tau) d\tau \Big)(p-p^*),
\end{equation}
where $\gamma_\tau = p^* + \tau (p-p^*)$,  $\tau \in [0,1]$, and $\gamma_\tau \in B(p^*, \varepsilon)$.
Let $A_p = \int_{0}^1{\nabla} H(\gamma_\tau) d\tau$ and $H_p = H(p^*) - A_p p^*$, so that $H(p) = A_p p + H_p$, where the subscript is used to indicate the dependence on $p$.
We will show that both $A_p$ and $H_p$ belong to compact sets, independent of $p$.

Regarding the matrix $A_p$, observe that for any $w \in {\mathbb{R}}^n$,
\begin{equation}
\label{eq:excitation_temp1}
    \|A_pw\|_2 \leq \int_0^1 \| {\nabla} H(\gamma_\tau)w \|_2d \tau \overset{(a)}{\leq} h_2 \|w\|_2,
\end{equation}
where $(a)$ holds by Lemma \ref{lemma:NE_bijection_mapping}. 
Also, for all $p\in B(p^*, \varepsilon)$,
\begin{equation}
\label{eq:excitation_temp2}
\operatorname{Sym}(RA_p)= \int_0^1 \operatorname{Sym}(R{\nabla} H(\gamma_\tau))d\tau.   
\end{equation}
Due to Lemma~\ref{lemma:NE_bijection_mapping}, map $H$ satisfies~\eqref{eq:first_order_condition} which further implies \(\nabla H(p) = - \nabla G_0(H(p))^{-1},\)
where $G_0(x) = \Theta^* \Phi(x)$ is the nominal pseudo-gradient.
Observe that
\begin{equation*}
\begin{aligned}
    &\operatorname{Sym}(R \nabla G_0(x)^{-1}) \\
     \overset{(a)}{=} &R \operatorname{Sym}((R \nabla G_0(x))^{-1}) R \\
     \overset{(b)}{=} &R (R\nabla G_0(x))^{-\top} \!\operatorname{Sym}(R \nabla G_0(x)) (R\nabla G_0(x))^{-1}R,
\end{aligned}
\end{equation*}
where $(a)$ is due to $R\nabla G_0(x)^{-1} = R(R\nabla G_0(x))^{-1}R$ and $(b)$ due to 
\(\operatorname{Sym}(M^{-1})=M^{-\top}\operatorname{Sym}(M)M^{-1}\).
By Assumption~\ref{ass:types}, 
\(
    \operatorname{Sym}\left(R\nabla G_0(x)\right)\succeq m/2
\), therefore
\[\begin{aligned}
    \operatorname{Sym}(R\nabla G_0(x)\!^{-1}) \!& \succeq\! \frac{m\lambda_{\min}^2(R)}{2}(R\nabla G_0(x))\!^{-\top}\!(R\nabla G_0(x))\!^{-1} \\
    & \succeq\frac{m\lambda_{\min}^2(R)h_1^2}{2\lambda_{\max}^2(R)} \mathbb I := a\mathbb I,
\end{aligned}\]
where $h_1\! =\! (\!\max_{x\in {\mathcal{X}}}\! \|\nabla G_0(x)\|_2)^{-1}$ as in Lemma~\ref{lemma:NE_bijection_mapping}, and the constant $a$ does not depend on $p$.
Hence, by~\eqref{eq:excitation_temp2},
\[\begin{aligned}
    \operatorname{Sym}(RA_p)  & = -\int_0^1 \operatorname{Sym}(R{\nabla} G_0(H(\gamma_\tau))^{-1})d\tau \preceq - a \mathbb I,
\end{aligned} \]
and therefore, for $w \in {\mathbb{R}}^n$,
\[a \|w\|_2^2 \leq - w^\top\! Sym(RA_p) w \leq \|R\|_2 \|A_pw\|_2\|w\|_2.\]
It follows that $\|A_pw\|_2 \geq( a / \|R\|_2) \|w\|_2$, This fact, along with~\eqref{eq:excitation_temp1}, implies that for all $p \in B(p^*, \varepsilon)$,
\[A_p \in \mathcal{A}:= \{A: \frac{a\|w\|_2}{\|R\|_2} \leq \|Aw\|_2 \leq h_2\|w\|_2,\,  \forall w \in {\mathbb{R}}^n\}.\]
Moreover, by~\eqref{eq:excitation_temp1}, $H_p$ satisfies
\(\|H_p\|_2\leq  \|H(p^*) \|_2 + h_2 \|p^*\|_2,\)
so it lies in a compact ball $\mathcal{B} \subset {\mathbb{R}}^n$.

We use Lemma \ref{lemma:M_matrix} on $H(p) = A_p p + H_p$ , to conclude that there is a $M_p \in {\mathbb{R}}^{D\times D}$ such that 
\begin{equation}
\label{eq:excitation_temp4}
    \begin{aligned}
    & \Phi(H(p))\Phi(H(p))^\top = M_p^\top \Phi(p) \Phi(p)^\top M_p,
\end{aligned} 
\end{equation}
where $M_p$ depends on $p$ and satisfies $\|M_p w\|_2 \geq \bar c\|w\|_2$, for a positive $\bar{c} > 0$, $ w \in {\mathbb{R}}^D$.
Since $\mathcal{A}$, $\mathcal{B}$ do not depend on $p$, the bound $\bar c$ is independent of $p$.
Taking the conditional expectation of \eqref{eq:excitation_temp4}, for all $w \in {\mathbb{R}}^D$, $\|w\|_2 = 1$,
\begin{align}
    w^\top Qw
    = &\mathbb E \big[  (\Phi(p)^\top M_pw)^2\mid p \in B(p^*, \varepsilon)\big], \label{eq:excitation_temp5}
\end{align}
where $\|M_p w \|_2 \geq \bar c\|w\|_2$ over the event $\{p \in B(p^*, \varepsilon)\}$.

We show that~\eqref{eq:excitation_temp5} implies $w^\top Q w > 0$ via contradiction.
Assume $w^\top Q w=0$. By~\eqref{eq:excitation_temp5},
\(\Phi(p)^\top M_p w = 0\), at almost every \(p\in B(p^*, \varepsilon).\)
Since $p \mapsto \Phi(p)^\top M_p w$ is continuous, any $p_0 \in B(p^*, \varepsilon)$ such that
$\Phi(p_0)^\top M_{p_0} w \neq 0$ would induce a nonempty neighborhood $U_{p_0}$, over which the integral~\eqref{eq:excitation_temp5} would be positive, so \(\Phi(H(p))=\Phi(p)^\top M_p w = 0\), for all \(p\in B(p^*, \varepsilon).\)
Lemma~\ref{lemma:PD_monomial_moments} applied to $H(p) \mid p \in B(p^*, \varepsilon)$
implies that $M_pw = 0$, which contradicts $\|M_p w \|_2 \geq \bar c\|w\|_2$.
\end{proof}

\subsection{Proof of Lemma \ref{lemma:eiv_excitation}}
\begin{proof}[\textbf{Proof of Lemma \ref{lemma:eiv_excitation}}]
   Similarly to Lemma \ref{lemma:excitation}, it is sufficient to show that
    \[\mathbb E (\xi^{i}(t)\xi^{j}(t)^{\top} \mid \delta(t)=1 ) \succ 0,\]
    where $i,j \in \{1,2\}$.
    On event $\{\delta(t)=1\}$, we have both $\hat x^{1}(t), \hat x^{2}(t) \in \operatorname{int} {\mathcal{X}}$ and thus $\hat x^{i}(t) = H(p(t) + v^{i}(t))$, $i\in \{1,2\}$, where $H$ is in Lemma~\ref{lemma:NE_bijection_mapping}.
    
    To ease notation, we omit the time dependence and write
    \[\xi^{i} = \Phi(\hat x^{i}) = \Phi(H(p+ v^{i})) \text{ and }\delta =1.\]
    We now consider the following two cases.
    \textit{Case 1:} $i \neq j$. Variable $\delta$ is measurable with respect to $(p,v^{3})$, and $v^{i}$, $ v^{j}$ are independent of $v^{3}$ and each other. Hence $v^{i}, v^{j}$ are independent when conditioned on $(p,\delta)$. 
    Then
        \[
    \begin{aligned}
    \mathbb E[\xi^i(\xi^j)^\top\!\mid\! p,\delta=1]
    &=\mathbb E_{v^{i}}[\xi^i\!\mid\! p,\delta=1]\,\mathbb E_{v^{j}}[\xi^j\!\mid\! p,\delta=1]^\top \\
    & = \mu_p \mu_{p}^\top,
    \end{aligned}
    \]
    where the LHS expectation is with respect to the joint law of $(v^{i}, v^{j})$, 
    \(\mu_p = \mathbb E_{v^{i}}[\xi^i\!\mid\! p,\delta=1] =\mathbb E_{v^{j}}[\xi^j\!\mid\! p,\delta=1],\)
    and the subscript denotes the dependence on $p$.

     \textit{Case 2:} $i = j$. By the convexity of $x \mapsto xx^\top$ and Jensen's inequality, we have
    \[
    \begin{aligned}
    \mathbb E[\xi^i(\xi^j)^\top\!\mid \!p,\delta=1]
    & \succeq\mathbb E_{v^{i}}[\xi^i\!\mid\! p,\delta=1]\,\mathbb E_{v^{j}}[\xi^j\!\mid\! p,\delta=1]^\top \\
    & = \mu_p \mu_p^\top,
    \end{aligned}
    \]
    where the LHS expectation is with respect to the joint law of $(v^{i}, v^{j})$, and
    $\mu_p$ as above.
    
    In both of the above cases, we have
    \begin{align}
    \mathbb E[\xi^i(\xi^j)^\top\!\mid\!\delta=1] &= \mathbb E_p \big[\mathbb E[\xi^i(\xi^j)^\top\!\mid\! p,\delta=1] \!\mid\! \delta = 1\big ] \notag \\
    & \succeq \mathbb E_p \big[ \mu_p \mu_p^\top \mid \delta = 1\big ]. \label{eq:eiv_excitation_temp1}
    \end{align}
    For any event \(V\subseteq\{\delta=1\}\) with positive probability,
    \begin{equation}
    \label{eq:eiv_excitation_temp2}
    \mathbb E_p[\mu_p \mu_p^\top \mid \delta=1]
    \succeq \mathbb P(V \mid \delta=1)\mathbb E_p[\mu_p \mu_p^\top \mid V].
    \end{equation}
    For $0< \varepsilon < (d_{\mathcal{X}}(x^\dagger) -\eta) / 2h_2$, define the event
    \[\begin{aligned}
        V &= \{ p  \in B(p^*, \varepsilon), \, \|v^{1}\|_2 \leq \varepsilon, \, \|v^{2}\|_2 \leq \varepsilon, \, \|v^{3}\|_2 \leq \varepsilon \} \\
        & \subset \bigcap_{k\in [3]} \{ p + v^{k} \in B(p^*, 2\varepsilon) \}.
    \end{aligned}\]
    Event $V$ has nonzero measure and for $q \in B(p^*, 2\varepsilon)$,
    \[ d_{\mathcal{X}}(H(q)) \geq d_{\mathcal{X}}(x^\dagger) - \|H(q)-x^\dagger\|_2 \geq d_{\mathcal{X}}(x^\dagger) - 2h_2\varepsilon \geq \eta,\]
    which implies that  $V \subset\{\delta(t)=1\}$.By~\eqref{eq:eiv_excitation_temp1} and~\eqref{eq:eiv_excitation_temp2}, it is 
    sufficient to show that
    \(Q = \mathbb E_p (\mu_p \mu_p^\top \mid V ) \succ 0.\)
   
    On $V$, we use a change-of-basis argument as in \eqref{eq:excitation_temp0}-\eqref{eq:excitation_temp4} from Lemma \ref{lemma:excitation}. For $i \in \{1,2\}$, define
    \begin{equation*}
    \begin{aligned}
        A_{p, v^{i}} & = \int_0^1 \nabla H(p^* + \tau (p+v^{i}-p^*)) d\tau,\\
        \text{and }b_{p, v^{i}} & = H(p^*) - A_{p, v^{i}} p^* + A_{p, v^{i}}v^{i},
    \end{aligned}
    \end{equation*}
    so $H(p+v^{i}) = H(p^*) + A_{p, v^{i}}(p+v^{i}-p^*)=  A_{p, v^{i}} p + b_{p, v^{i}}$.
    Repeating the bounds derived in Lemma~\ref{lemma:excitation}, with $p$ replaced by $p+v^{i}$, together with the boundedness $\|v^{i}\|_2 \leq \bar v\sqrt{n}$, yields compact sets $\mathcal{A}$, $\mathcal{B}$. Hence, by Lemma~\ref{lemma:M_matrix},
    \begin{equation*}
    \label{eq:eiv_excitation_temp3}
        \begin{aligned}
        & \mu_p\mu_p^\top =  \mathbb E_{v_i}[M_{p,v^{i}}\mid V, p]^\top \Phi(p) \Phi(p)^\top \mathbb E_{v_i}[M_{p,v^{i}}\mid V,p],
    \end{aligned} 
    \end{equation*}
    where $M_{p,v^{i}} \in {\mathbb{R}}^{D\times D}$ depends on $(p, v^{i})$ and satisfies $\|M_{p,v^{i}} w \|_2 \geq \bar{c}\|w\|_2$, for a $\bar{c} > 0$, all $w\in {\mathbb{R}}^D$, and $p+v^{i}\in B(p^*, 2\varepsilon)$. 
    Denote $\bar M_p = \mathbb E_{v_i}[M_{p,v^{i}}\mid V,p]$, so 
    \begin{align}
    w^\top Qw
    = &\mathbb E \big[  (\Phi(p)^\top \bar M_pw)^2\mid p \in B(p^*, \varepsilon)\big]. \label{eq:eiv_excitation_temp4}
\end{align}
 Assume $w^\top Q w = 0$. By~\eqref{eq:eiv_excitation_temp4} and the continuity of $p \mapsto \Phi(p)^\top \bar{M}_p w$, it follows that for all $p \in B(p^*, \varepsilon)$,
\[\begin{aligned}
    \Phi(p)^\top \bar{M}_p w
    = \mathbb E_{v_i}[\Phi(p)^\top M_{p,v^{i}}w\mid V,p]
    = 0.
\end{aligned}\]
In turn, the above implies
\[\int_{\|v\|_2\leq \varepsilon}\Phi(H(p+v))^\top w f_v(v) dv = 0,\ \forall p\in B(p^*, \varepsilon),\]
where $f_v$ denotes the density of $v$. This implies that 
\(\Phi(H(q))^\top w = \Phi(p)^\top M_{p,v}w= 0\) over $q = p+v\in B(p^*, 2\varepsilon)$. 
Lemma~\ref{lemma:PD_monomial_moments} then  implies $M_{p,v}w= 0$, which contradicts 
$\|M_{p,v^{i}} w \|_2 \geq \bar{c}\|w\|_2$.
\end{proof}

\vspace{-3mm}
\begin{ack}
\vspace{-3mm}
This work was partially supported by the Wallenberg AI, Autonomous Systems and Software Program (WASP), funded by the Knut and Alice Wallenberg Foundation.
\end{ack}

\bibliographystyle{IEEEtran}      
\bibliography{bibliography}     

@article{basarAffineIncentiveSchemes1984,
  title = {Affine {{Incentive Schemes}} for {{Stochastic Systems}} with {{Dynamic Information}}},
  author = {Ba{\c s}ar, Tamer},
  year = 1984,
  month = mar,
  journal = {SIAM Journal on Control and Optimization},
  volume = {22},
  number = {2},
  pages = {199--210},
  issn = {0363-0129, 1095-7138},
  doi = {10.1137/0322015},
  urldate = {2025-04-19},
  langid = {english}
}

@article{basarClosedloopStackelbergStrategies1979a,
  title = {Closed-Loop {{Stackelberg}} Strategies with Applications in the Optimal Control of Multilevel Systems},
  author = {Basar, T. and Selbuz, H.},
  year = 1979,
  month = apr,
  journal = {IEEE Transactions on Automatic Control},
  volume = {24},
  number = {2},
  pages = {166--179},
  issn = {0018-9286},
  doi = {10.1109/TAC.1979.1101999},
  urldate = {2024-07-05},
  copyright = {https://ieeexplore.ieee.org/Xplorehelp/downloads/license-information/IEEE.html},
  langid = {english}
}

@inproceedings{grootReverseStackelbergGames2012,
  title = {Reverse {{Stackelberg}} Games, {{Part I}}: {{Basic}} Framework},
  shorttitle = {Reverse {{Stackelberg}} Games, {{Part I}}},
  booktitle = {the 2012 {{International Conference}} on {{Control Applications}}},
  author = {Groot, Noortje and De Schutter, Bart and Hellendoorn, Hans},
  year = 2012,
  month = oct,
  pages = {421--426},
  address = {Croatia},
  doi = {10.1109/CCA.2012.6402334},
  urldate = {2025-04-22},
  isbn = {978-1-4673-4505-7 978-1-4673-4503-3 978-1-4673-4504-0}
}

@inproceedings{grootReverseStackelbergGames2012a,
  title = {Reverse {{Stackelberg}} Games, Part {{II}}: {{Results}} and Open Issues},
  shorttitle = {Reverse {{Stackelberg}} Games, Part {{II}}},
  booktitle = {Proc. of the 2012 {{ International Conference}} on {{Control Applications}}},
  author = {Groot, Noortje and De Schutter, Bart and Hellendoorn, Hans},
  year = 2012,
  month = oct,
  pages = {427--432},
  address = {Dubrovnik, Croatia},
  doi = {10.1109/CCA.2012.6402335},
  urldate = {2025-10-27},
  isbn = {978-1-4673-4505-7 978-1-4673-4503-3 978-1-4673-4504-0}
}

@article{grootSystemOptimalRoutingTraffic2015,
  title = {Toward {{System-Optimal Routing}} in {{Traffic Networks}}: {{A Reverse Stackelberg Game Approach}}},
  shorttitle = {Toward {{System-Optimal Routing}} in {{Traffic Networks}}},
  author = {Groot, Noortje and De Schutter, Bart and Hellendoorn, Hans},
  year = 2015,
  month = feb,
  journal = {IEEE Transactions on Intelligent Transportation Systems},
  volume = {16},
  number = {1},
  pages = {29--40},
  issn = {1524-9050, 1558-0016},
  doi = {10.1109/TITS.2014.2322312},
  urldate = {2025-06-05},
  copyright = {https://ieeexplore.ieee.org/Xplorehelp/downloads/license-information/IEEE.html}
}

@inproceedings{hoControltheoreticViewIncentives1980,
  title = {A Control-Theoretic View on Incentives},
  booktitle = {Proc. of the 19th {{ Conference}} on {{Decision}} and {{Control}} (CDC)},
  author = {Ho, Yu-chi and Luh, Peter and Olsder, Geert},
  year = 1980,
  month = dec,
  pages = {1160--1170},
  address = {Albuquerque, NM, USA},
  doi = {10.1109/CDC.1980.271986},
  urldate = {2025-04-19}
}

@article{laiLeastSquaresEstimates1982,
  title = {Least {{Squares Estimates}} in {{Stochastic Regression Models}} with {{Applications}} to {{Identification}} and {{Control}} of {{Dynamic Systems}}},
  author = {Lai, Tze Leung and Wei, Ching Zong},
  year = 1982,
  month = mar,
  journal = {The Annals of Statistics},
  volume = {10},
  number = {1},
  issn = {0090-5364},
  doi = {10.1214/aos/1176345697},
  urldate = {2025-04-23}
}

@inproceedings{liAdaptivePricingOptimal2025,
  author={Li, Jiayi and Wei, Jiale and Motoki, Matthew and Jiang, Yan and Zhang, Baosen},
  booktitle={Proc. of the 64th Conference on Decision and Control (CDC)}, 
  title={Adaptive Pricing for Optimal Coordination in Networked Energy Systems with Nonsmooth Cost Functions}, 
  year={2025},
  volume={},
  number={},
  pages={4043-4050},
  doi={10.1109/CDC57313.2025.11312710}}

@article{liSociallyOptimalEnergy2024,
  title = {Socially Optimal Energy Usage via Adaptive Pricing},
  author = {Li, Jiayi and Motoki, Matthew and Zhang, Baosen},
  year = 2024,
  month = oct,
  journal = {Electric Power Systems Research},
  volume = {235},
  pages = {110640},
  issn = {03787796},
  doi = {10.1016/j.epsr.2024.110640},
  urldate = {2025-04-21},
  langid = {english}
}

@article{maheshwariAdaptiveIncentiveDesign2024,
  author={Maheshwari, Chinmay and Kulkarni, Kshitij and Wu, Manxi and Sastry, Shankar},
  journal={IEEE Transactions on Automatic Control}, 
  title={Adaptive Incentive Design With Learning Agents}, 
  year={2025},
  volume={},
  number={},
  pages={1-16},
  doi={10.1109/TAC.2025.3643351}}

@article{picardDesignIncentiveSchemes1987,
  title = {On the Design of Incentive Schemes under Moral Hazard and Adverse Selection},
  author = {Picard, Pierre},
  year = 1987,
  month = aug,
  journal = {Journal of Public Economics},
  volume = {33},
  number = {3},
  pages = {305--331},
  issn = {00472727},
  doi = {10.1016/0047-2727(87)90058-2},
  urldate = {2025-04-19},
  copyright = {https://www.elsevier.com/tdm/userlicense/1.0/},
  langid = {english}
}

@article{ratliffAdaptiveIncentiveDesign2021,
  title = {Adaptive {{Incentive Design}}},
  author = {Ratliff, Lillian J. and Fiez, Tanner},
  year = 2021,
  month = aug,
  journal = {IEEE Transactions on Automatic Control},
  volume = {66},
  number = {8},
  pages = {3871--3878},
  issn = {0018-9286, 1558-2523, 2334-3303},
  doi = {10.1109/TAC.2020.3027503},
  urldate = {2024-02-24}
}

@article{ratliffPerspectiveIncentiveDesign2019,
  title = {A {{Perspective}} on {{Incentive Design}}: {{Challenges}} and {{Opportunities}}},
  shorttitle = {A {{Perspective}} on {{Incentive Design}}},
  author = {Ratliff, Lillian J. and Dong, Roy and Sekar, Shreyas and Fiez, Tanner},
  year = 2019,
  month = may,
  journal = {Annual Review of Control, Robotics, and Autonomous Systems},
  volume = {2},
  number = {1},
  pages = {305--338},
  issn = {2573-5144, 2573-5144},
  doi = {10.1146/annurev-control-053018-023634},
  urldate = {2024-05-13},
  langid = {english}
}

@article{simaanStackelbergSolutionGames1973,
  title = {A {{Stackelberg}} Solution for Games with Many Players},
  author = {Simaan, M. and Cruz, J.},
  year = 1973,
  month = jun,
  journal = {IEEE Trans. Autom. Control},
  volume = {18},
  number = {3},
  pages = {322--324},
  issn = {0018-9286},
  doi = {10.1109/TAC.1973.1100307},
  urldate = {2024-05-30},
  copyright = {https://ieeexplore.ieee.org/Xplorehelp/downloads/license-information/IEEE.html},
  langid = {english}
}

@article{tolwinskiClosedloopStackelbergSolution1981,
  title = {Closed-Loop {{Stackelberg}} Solution to a Multistage Linear-Quadratic Game},
  author = {Tolwinski, B.},
  year = 1981,
  month = aug,
  journal = {Journal of Optimization Theory and Applications},
  volume = {34},
  number = {4},
  pages = {485--501},
  issn = {0022-3239, 1573-2878},
  doi = {10.1007/BF00935889},
  urldate = {2024-05-30},
  copyright = {http://www.springer.com/tdm},
  langid = {english}
}

@inproceedings{yorulmazSoftInducementFramework2025,
  author={Yorulmaz, Asrin Efe and Velicheti, Raj Kiriti and Bastopcu, Melih and Başar, Tamer},
  booktitle={Proc. of the 64th Conference on Decision and Control (CDC)}, 
  title={A Soft Inducement Framework for Incentive-Aided Steering of No-Regret Players}, 
  year={2025},
  volume={},
  number={},
  pages={4396-4401},
  address = {Rio de Janeiro, Brazil}
  }

@book{boltonContractTheory2005,
  title = {Contract Theory},
  author = {Bolton, Patrick and Dewatripont, M. and Campbell, Arthur},
  year = {2005},
  publisher = {MIT Press},
  location = {Cambridge, Mass},
  isbn = {978-0-262-02576-8 978-0-262-53299-0},
  pagetotal = {724}
}

@ARTICLE{barreraDynamicIncentives,
  author={Barrera, Jorge and Garcia, Alfredo},
  journal={IEEE Transactions on Automatic Control}, 
  title={Dynamic Incentives for Congestion Control}, 
  year={2015},
  volume={60},
  number={2},
  pages={299-310},
  doi={10.1109/TAC.2014.2348197}}

@inproceedings{povedaDistributedAdaptivePricing,
  author={Poveda, Jorge I. and Brown, Philip N. and Marden, Jason R. and Teel, Andrew R.},
  booktitle={the 56th Conference on Decision and Control}, 
  title={A class of distributed adaptive pricing mechanisms for societal systems with limited information}, 
  year={2017},
  volume={},
  number={},
  pages={1490-1495},
  doi={10.1109/CDC.2017.8263863}}

@article{rosenExistenceUniquenessEquilibrium1965a,
  title = {Existence and {{Uniqueness}} of {{Equilibrium Points}} for {{Concave N-Person Games}}},
  author = {Rosen, J. B.},
  year = 1965,
  month = jul,
  journal = {Econometrica},
  volume = {33},
  number = {3},
  pages = {520},
  issn = {00129682},
  doi = {10.2307/1911749},
  urldate = {2025-12-04},
}

@article{hoInformationStructure1981,
  author={Yu-Chi Ho and Luh, P. and Muralidharan, R.},
  journal={IEEE Transactions on Automatic Control}, 
  title={Information structure, Stackelberg games, and incentive controllability}, 
  year={1981},
  volume={26},
  number={2},
  pages={454-460},
  doi={10.1109/TAC.1981.1102652}}

@incollection{zhangSolutionStackelbergMultifollower,
  title={The solution to a kind of Stackelberg game systems with multi-follower: Coordinative and incentive},
  author={Jing, Yuan-wei and Zhang, Si-ying},
  booktitle={Analysis and optimization of systems},
  pages={593--602},
  year={2006},
  publisher={Springer}
}

@article{cruzLeaderfollowerStrategiesMultilevel1978,
  title = {Leader-Follower Strategies for Multilevel Systems},
  author = {Cruz, J.},
  year = 1978,
  month = apr,
  journal = {IEEE Transactions on Automatic Control},
  volume = {23},
  number = {2},
  pages = {244--255},
  issn = {0018-9286},
  doi = {10.1109/TAC.1978.1101716},
  urldate = {2026-01-19},
  copyright = {https://ieeexplore.ieee.org/Xplorehelp/downloads/license-information/IEEE.html},
  langid = {english}
}

@article{zhengExistenceDerivationOptimal1982a,
  title = {Existence and Derivation of Optimal Affine Incentive Schemes for {{Stackelberg}} Games with Partial Information: A Geometric Approach},
  shorttitle = {Existence and Derivation of Optimal Affine Incentive Schemes for {{Stackelberg}} Games with Partial Information},
  author = {Zheng, Ying-Ping and Basar, Tamer},
  year = 1982,
  month = jun,
  journal = {International Journal of Control},
  volume = {35},
  number = {6},
  pages = {997--1011},
  urldate = {2026-01-19},
  langid = {english}
}

@article{luhCredibilityStackelbergGames1984,
  title = {Credibility in Stackelberg Games},
  author = {Luh, Peter B. and Zheng, Ying-Ping and Ho, Yu-Chi},
  year = 1984,
  month = dec,
  journal = {Systems \& Control Letters},
  volume = {5},
  number = {3},
  pages = {165--168},
  issn = {01676911},
  doi = {10.1016/S0167-6911(84)80098-5},
  urldate = {2026-01-19},
  copyright = {https://www.elsevier.com/tdm/userlicense/1.0/},
  langid = {english}
}

@article{xiaopingliuOptimalIncentiveStrategy1992,
  title = {Optimal Incentive Strategy for Leader-Follower Games},
  author = {Liu, Xiaoping and Zhang, Siying},
  year = 1992,
  month = dec,
  journal = {IEEE Transactions on Automatic Control},
  volume = {37},
  number = {12},
  pages = {1957--1961},
  issn = {00189286},
  doi = {10.1109/9.182481},
  urldate = {2025-10-30},
  copyright = {https://ieeexplore.ieee.org/Xplorehelp/downloads/license-information/IEEE.html}
}

@article{tuPerformanceInformativenessLinearquadratic1988,
  title = {Performance versus Informativeness in Linear-Quadratic {{Gaussian}} Noncooperative Games},
  author = {Tu, M. and Papavassilopoulos, G. P.},
  year = 1988,
  month = apr,
  journal = {Journal of Optimization Theory and Applications},
  volume = {57},
  number = {1},
  pages = {161--187},
  issn = {0022-3239, 1573-2878},
  doi = {10.1007/BF00939334},
  urldate = {2026-01-19},
  copyright = {http://www.springer.com/tdm},
  langid = {english}
}

@article{laufferNoRegretLearningDynamic2024,
  title = {No-{{Regret Learning}} in {{Dynamic Stackelberg Games}}},
  author = {Lauffer, Niklas and Ghasemi, Mahsa and Hashemi, Abolfazl and Savas, Yagiz and Topcu, Ufuk},
  year = 2024,
  month = mar,
  journal = {IEEE Transactions on Automatic Control},
  volume = {69},
  number = {3},
  pages = {1418--1431},
  issn = {0018-9286, 1558-2523, 2334-3303},
  doi = {10.1109/TAC.2023.3330797},
  urldate = {2026-01-19},
  copyright = {https://ieeexplore.ieee.org/Xplorehelp/downloads/license-information/IEEE.html}
}

@book{fudenbergGameTheory1991,
  title = {Game Theory},
  author = {Fudenberg, Drew and Tirole, Jean},
  year = 1991,
  publisher = {MIT press},
  address = {Cambridge Massachusetts},
  isbn = {978-0-262-06141-4},
  langid = {english},
  lccn = {330.015 193}
}

@article{olsderPhenomenaInverseStackelberg2009,
  title = {Phenomena in {{Inverse Stackelberg Games}}, {{Part}} 1: {{Static Problems}}},
  shorttitle = {Phenomena in {{Inverse Stackelberg Games}}, {{Part}} 1},
  author = {Olsder, G. J.},
  year = 2009,
  month = dec,
  journal = {Journal of Optimization Theory and Applications},
  volume = {143},
  number = {3},
  pages = {589--600},
  issn = {0022-3239, 1573-2878},
  doi = {10.1007/s10957-009-9573-9},
  urldate = {2026-01-19},
  langid = {english}
}

@article{olsderPhenomenaInverseStackelberg2009a,
  title = {Phenomena in {{Inverse Stackelberg Games}}, {{Part}} 2: {{Dynamic Problems}}},
  shorttitle = {Phenomena in {{Inverse Stackelberg Games}}, {{Part}} 2},
  author = {Olsder, G. J.},
  year = 2009,
  month = dec,
  journal = {Journal of Optimization Theory and Applications},
  volume = {143},
  number = {3},
  pages = {601--618},
  issn = {0022-3239, 1573-2878},
  doi = {10.1007/s10957-009-9572-x},
  urldate = {2026-01-19},
  langid = {english}
}

@article{calderonePricingCoordinationOpen2014,
  title = {Pricing for {{Coordination}} in {{Open}}--{{Loop Differential Games}}},
  author = {Calderone, Daniel and Ratliff, Lillian J. and Sastry, S. Shankar},
  year = 2014,
  journal = {IFAC Proceedings Volumes},
  volume = {47},
  number = {3},
  pages = {9001--9006},
  issn = {14746670},
  doi = {10.3182/20140824-6-ZA-1003.02655},
  urldate = {2026-01-19},
  copyright = {https://www.elsevier.com/tdm/userlicense/1.0/},
  langid = {english}
}

@inproceedings{balcanCommitmentRegretsOnline2015a,
  title = {Commitment {{Without Regrets}}: {{Online Learning}} in {{Stackelberg Security Games}}},
  shorttitle = {Commitment {{Without Regrets}}},
  booktitle = {Proc. of the {{16th ACM Conference}} on {{Economics}} and {{Computation}}},
  author = {Balcan, Maria-Florina and Blum, Avrim and Haghtalab, Nika and Procaccia, Ariel D.},
  year = 2015,
  month = jun,
  pages = {61--78},
  address = {Portland Oregon USA},
  doi = {10.1145/2764468.2764478},
  urldate = {2026-01-19},
  isbn = {978-1-4503-3410-5},
  langid = {english}
}

@inproceedings{harrisRegretMinimizationStackelberg2024,
author = {Harris, Keegan and Wu, Zhiwei Steven and Balcan, Maria-Florina},
title = {Regret minimization in stackelberg games with side information},
year = {2024},
isbn = {9798331314385},
booktitle = {Proc. of the 38th International Conference on Neural Information Processing Systems (NIPS)},
articleno = {412},
numpages = {33},
address = {Vancouver, BC, Canada},
}

@article{afriatConstructionUtilityFunctions1967,
  title = {The {{Construction}} of {{Utility Functions}} from {{Expenditure Data}}},
  author = {Afriat, S. N.},
  year = 1967,
  month = feb,
  journal = {International Economic Review},
  volume = {8},
  number = {1},
  pages = {67},
  issn = {00206598},
  doi = {10.2307/2525382},
  urldate = {2026-01-19}
}

@inproceedings{kuleshovInverseGameTheory2015,
  title={Inverse game theory: Learning utilities in succinct games},
  author={Kuleshov, Volodymyr and Schrijvers, Okke},
  booktitle={International Conference on Web and Internet Economics},
  pages={413--427},
  year={2015},
  organization={Springer}
}

@article{varianNonparametricApproachDemand1982,
  title = {The {{Nonparametric Approach}} to {{Demand Analysis}}},
  author = {Varian, Hal R.},
  year = 1982,
  month = jul,
  journal = {Econometrica},
  volume = {50},
  number = {4},
  pages = {945},
  issn = {00129682},
  doi = {10.2307/1912771},
  urldate = {2026-01-19}
}

@incollection{varianRevealedPreference2006,
  title = {Revealed {{Preference}}},
  booktitle = {Samuelsonian {{Economics}} and the {{Twenty-First Century}}},
  author = {Varian, Hal R.},
  year = 2006,
  month = aug,
  pages = {99--115},
  publisher = {Oxford University Press},
  langid = {english}
}

@article{liseEstimatingGameTheoretic2001,
  title = {Estimating a {{Game Theoretic Model}}},
  author = {Lise, Wietze},
  year = 2001,
  month = oct,
  journal = {Computational Economics},
  volume = {18},
  number = {2},
  pages = {141--157},
  issn = {0927-7099, 1572-9974},
  doi = {10.1023/A:1021086215235},
  urldate = {2026-01-19},
  copyright = {https://www.springernature.com/gp/researchers/text-and-data-mining},
  langid = {english}
}

@inproceedings{ngAlgorithmsInverseReinforcementLearning,
author = {Ng, Andrew Y. and Russell, Stuart J.},
title = {Algorithms for Inverse Reinforcement Learning},
year = {2000},
isbn = {1558607072},
address = {San Francisco, CA, USA},
booktitle = {the 17th International Conference on Machine Learning (ICML)},
pages = {663–670},
numpages = {8},
}

@inproceedings{sessaContextualGamesMultiAgent2021,
author = {Sessa, Pier Giuseppe and Bogunovic, Ilija and Krause, Andreas and Kamgarpour, Maryam},
title = {Contextual games: multi-agent learning with side information},
year = {2020},
isbn = {9781713829546},
booktitle = {the 34th International Conference on Neural Information Processing Systems (NIPS)},
articleno = {1838},
numpages = {11},
address = {Canada},
}

@unpublished{vasileiou2026adaptiveincentivedesignregret,
      title={Adaptive Incentive Design with Regret Minimization}, 
      author={Georgios Vasileiou and Lantian Zhang and Silun Zhang},
      year={2026},
      note ={preprint arxiv:2604.05977},
      archivePrefix={arXiv}
}

@unpublished{vasileiou2026incentivedesignhypergradientssocialgradient,
      title={Incentive Design without Hypergradients: A Social-Gradient Method}, 
      author={Georgios Vasileiou and Lantian Zhang and Silun Zhang},
      year={2026},
      note={arxiv:2604.11346},

}

@inproceedings{changConceptInducibleRegion1982,
  title = {The {{Concept}} of {{Inducible Region}} in {{Stackelberg Games}}},
  booktitle = {Proc. of the 1982 {{American Control Conference}}},
  author = {Chang, Tsu-Shuan and Luh, Peter B.},
  year = 1982,
  month = jun,
  pages = {139--140},
  address = {Arlington, VA, USA},
  doi = {10.23919/ACC.1982.4787820},
  urldate = {2026-05-05}
}

@inproceedings{hoCredibilityRationalityPlayers1982,
  title = {Credibility and Rationality of Players Strategies in Multilevel Games},
  booktitle = {21st {{Conference}} on {{Decision}} and {{Control}} (CDC)},
  author = {Ho, Y. C. and Tolwinski, B.},
  year = 1982,
  month = dec,
  pages = {659--663},
  doi = {10.1109/CDC.1982.268223},
  urldate = {2026-05-05}
}

@article{liNonparametricEstimationMeasurement1998,
  title = {Nonparametric {{Estimation}} of the {{Measurement Error Model Using Multiple Indicators}}},
  author = {Li, Tong and Vuong, Quang},
  year = 1998,
  month = may,
  journal = {Journal of Multivariate Analysis},
  volume = {65},
  number = {2},
  pages = {139--165},
  issn = {0047259X},
  doi = {10.1006/jmva.1998.1741},
  urldate = {2026-05-07},
  copyright = {https://www.elsevier.com/tdm/userlicense/1.0/},
  langid = {english}
}

@article{hausmanIdentificationEstimationPolynomial1991,
  title = {Identification and Estimation of Polynomial Errors-in-Variables Models},
  author = {Hausman, Jerry A. and Newey, Whitney K. and Ichimura, Hidehiko and Powell, James L.},
  year = 1991,
  month = dec,
  journal = {Journal of Econometrics},
  volume = {50},
  number = {3},
  pages = {273--295},
  issn = {03044076},
  doi = {10.1016/0304-4076(91)90022-6},
  urldate = {2026-05-07},
  copyright = {https://www.elsevier.com/tdm/userlicense/1.0/},
  langid = {english}
}

@ARTICLE{umar26ecodriving,
  author={Niazi, M. Umar B. and Cho, Jung-Hoon and Dahleh, Munther A. and Dong, Roy and Wu, Cathy},
  journal={IEEE Transactions on Control of Network Systems}, 
  title={Eco-Driving Incentive Mechanisms for Mitigating Emissions in Urban Transportation}, 
  year={2026},
  volume={13},
  number={1},
  pages={166-178},
  doi={10.1109/TCNS.2025.3623943}}

@ARTICLE{satchidanandan23atwostage,
  author={Satchidanandan, Bharadwaj and Roozbehani, Mardavij and Dahleh, Munther A.},
  journal={IEEE Control Systems Letters}, 
  title={A Two-Stage Mechanism for Demand Response Markets}, 
  year={2023},
  volume={7},
  number={},
  pages={49-54},
  doi={10.1109/LCSYS.2022.3186654}}

@article{bonatti24sellinginfo,
title = {Selling information in competitive environments},
journal = {Journal of Economic Theory},
volume = {216},
pages = {105779},
year = {2024},
issn = {0022-0531},
author = {Alessandro Bonatti and Munther Dahleh and Thibaut Horel and Amir Nouripour}
}

@misc{zhang2026stochasticadaptivecontrolsystems,
      title={Stochastic Adaptive Control for Systems with Nonlinear Parameterization: Almost Sure Stability and Tracking}, 
      author={Lantian Zhang and Bo Wahlberg and Silun Zhang},
      year={2026},
      note={preprint arXiv:2604.06980},
}

@inproceedings{vanweerelt2025selfidentifyinginternalmodelbasedonline,
  title = {Self-Identifying Internal Model-Based Online Optimization},
  booktitle = {IFAC},
  author={Wouter J. A. van Weerelt and Lantian Zhang and Silun Zhang and Nicola Bastianello},
  year = 2026,
  month = jul
}

@misc{zhang2025onlinelearningnonlineardynamical,
      title={Online Learning for Nonlinear Dynamical Systems without the I.I.D. Condition}, 
      author={Lantian Zhang and Silun Zhang},
      year={2025},
      note={preprint arXiv:2504.02995}
}

@misc{chen2025activeinversemethodsstackelberg,
      title={Active Inverse Methods in Stackelberg Games with Bounded Rationality}, 
      author={Jianguo Chen and Jinlong Lei and Biqiang Mu and Yiguang Hong and Hongsheng Qi},
      year={2025},
      archivePrefix={arXiv},
      primaryClass={cs.GT}
}

@ARTICLE{zhang21modelingcollective,
  author={Zhang, Silun and Ringh, Axel and Hu, Xiaoming and Karlsson, Johan},
  journal={IEEE Transactions on Automatic Control}, 
  title={Modeling Collective Behaviors: A Moment-Based Approach}, 
  year={2021},
  volume={66},
  number={1},
  pages={33-48},}

@article{zhang20anintrinsicapproach,
title = {An intrinsic approach to formation control of regular polyhedra for reduced attitudes},
journal = {Automatica},
volume = {111},
pages = {108619},
year = {2020},
author = {Silun Zhang and Fenghua He and Yiguang Hong and Xiaoming Hu}
}

@article{zhang18intrinsictetrahedron,
title = {Intrinsic tetrahedron formation of reduced attitude},
journal = {Automatica},
volume = {87},
pages = {375-382},
year = {2018},
author = {Silun Zhang and Wenjun Song and Fenghua He and Yiguang Hong and Xiaoming Hu}
}

\end{document}